\documentclass[review,onefignum,onetabnum]{siamonline220329}

\usepackage{xcolor}
\usepackage{colortbl}
\usepackage{algorithm,bbding,pifont}
\usepackage{algorithm,algorithmic}

\definecolor{olive}{rgb}{0.6, 0.6, 0.2}
\definecolor{sand}{rgb}{0.8666666666666667, 0.8, 0.4666666666666667}
\definecolor{wine}{rgb}{0.5333333333333333, 0.13333333333333333, 0.3333333333333333}
\usepackage{MnSymbol,bbding,pifont}
\usepackage{wrapfig}

\usepackage{fontawesome5}

\definecolor{icmlblue}{rgb}{0.21,0.49,0.74}
\newsiamremark{assumption}{Assumption}

\definecolor{electricindigo}{rgb}{0.44, 0.0, 1.0}
\usepackage{multirow}

\usepackage{graphicx}
\usepackage{tikz}
\usepackage{subfigure}
\usetikzlibrary{positioning}
\usetikzlibrary{calc} 
\usepackage{ifthen}
\usepackage{xparse} 

 \ExplSyntaxOn
\NewDocumentCommand{\aspectratio}{smo}
 {
  \hbox_set:Nn \l_tmpa_box {\includegraphics{#2}}
  \IfNoValueTF{#3}
   {
    \__student_aspectratio:nn { \box_wd:N \l_tmpa_box } { \box_ht:N \l_tmpa_box }
   }
   {
    \IfBooleanTF{#1}{ \tl_gset:Nx } { \tl_set:Nx } #3
     {
      \__student_aspectratio:nn { \box_wd:N \l_tmpa_box } { \box_ht:N \l_tmpa_box }
     }
   }
 }

\cs_new:Nn \__student_aspectratio:nn
 {
  \fp_eval:n {round( #1 / #2 , 5)}
 }
\ExplSyntaxOff

\newcommand{\neworrenewcommand}[1]{\providecommand{#1}{}\renewcommand{#1}}

\newcommand{\zoomincludgraphic}[9]{
    \neworrenewcommand{\ffoo}[5]{
\begin{tikzpicture}[x=#1, y=#1, font=\footnotesize]
\aspectratio{#2}[\imsizeratio] 

  \node[anchor = south east, inner sep=0] (image) at (1,0) {\includegraphics[width=#1]{#2}};
	    \coordinate (viewport lower left) at (#3,#4/\imsizeratio);  
	    \coordinate(viewport upper right) at (#5,#6/\imsizeratio);  
        \draw[##5, line width = ##4 pt] (viewport lower left) rectangle (viewport upper right);
 
     \pgfmathsetmacro{\multone}{#5-#3}
     \pgfmathsetmacro{\multtwo}{#6/\imsizeratio-#4/\imsizeratio}
     
     \ifthenelse{\equal{#9}{bottom_left} }{ 
	      \node[anchor=north, draw= ##3, inner sep=0pt, line width = ##2 pt,outer sep=0pt] (zoomPart) at (\multone*#7/2+##2/345*1.333, \multtwo*#7+##2/345*1.333) {
	       \scalebox{#7}{\tikz{
	         \clip (#3,#4/\imsizeratio) rectangle (#5,#6/\imsizeratio);
	           
	         \node[anchor=south east, inner sep=0] at (1,0) {\includegraphics[width=#1]{#2}}; 
	         }}};
	   \ifthenelse{\equal{##1}{line_connection_on} }{ 
		  \draw[red, dashed] (viewport upper right|-viewport lower left) -- (zoomPart.north east); 
		  \draw[red, dashed] (viewport lower left) -- (zoomPart.north west);
		   }{}
	       
	 }{}

     \ifthenelse{\equal{#9}{bottom_right} }{ 
	      \node[anchor=north, draw= ##3, inner sep=0pt, line width = ##2 pt,outer sep=0pt] (zoomPart) at (1-\multone*#7/2-##2/345*1.333, \multtwo*#7 + ##2/345*1.333) {
	       \scalebox{#7}{\tikz{
			 \clip (#3,#4/\imsizeratio) rectangle (#5,#6/\imsizeratio);
	         \node[anchor=south east, inner sep=0] at (1,0) {\includegraphics[width=#1]{#2}}; 
	         }}%
	       };
		\ifthenelse{\equal{##1}{line_connection_on} }{ 
			  \draw[red, dashed] (viewport upper right|-viewport lower left) -- (zoomPart.south west); 
			  \draw[red, dashed] (viewport upper right) -- (zoomPart.north west);
		   }{}
     }{}  
       
    \ifthenelse{\equal{#9}{up_right} }{  
	      \node[anchor=north, draw= ##3, inner sep=0pt, line width = ##2pt, outer sep=0pt] (zoomPart) at (1-\multone*#7/2-##2/345*1.333,1/\imsizeratio-##2/345*1.333) {
	       \scalebox{#7}{\tikz{
	          \clip (#3,#4/\imsizeratio) rectangle (#5,#6/\imsizeratio);
	          \node[anchor=south east, inner sep=0] at (1,0) {\includegraphics[width=#1]{#2}}; 
	         }}%
	         
	       };
	   \ifthenelse{\equal{##1}{line_connection_on} }{ 
		  \draw[red, dashed] (viewport lower left|-viewport upper right) -- (zoomPart.south west);
		  \draw[red, dashed] (viewport upper right) -- (zoomPart.south east);
		   }{}
     }{}

     \ifthenelse{\equal{#9}{up_left} }{ 
	      \node[anchor=north, draw= ##3, inner sep=0pt, line width = ##2pt,outer sep=0pt] (zoomPart) at (\multone*#7/2+##2/345*1.333, 1/\imsizeratio-##2/345*1.333) {
	       \scalebox{#7}{\tikz{
	         \clip (#3,#4/\imsizeratio) rectangle (#5,#6/\imsizeratio);
	           
	          \node[anchor=south east, inner sep=0] at (1,0) {\includegraphics[width=#1]{#2}}; 
	         }}};
		   \ifthenelse{\equal{##1}{line_connection_on} }{ 
			  \draw[red, dashed] (viewport lower left|-viewport upper right) -- (zoomPart.south west);
			  \draw[red, dashed] (viewport upper right) -- (zoomPart.south east);
			   }{}
	     }{}

  	\ifthenelse{\equal{#8}{help_grid_on} }{ 
           \begin{scope}[
                x={(image.south east)},
                y={(image.north west)},
                font=\footnotesize,
                help lines,
                overlay
            ]
            
            \draw[help lines, xstep=.1,ystep=.1,overlay] (0,0) grid (1,1);
            \foreach \x in {0,1,...,9} { 
                \node[anchor=north] at (\x/10,0) {0.\x}; 
            }
            \foreach \y in {0,1,...,9} {
                \node[anchor=east] at (0,\y/10) {0.\y};
            }
        \end{scope}    
	}{}  
   
\end{tikzpicture}

    }
    \ffoo
}

\definecolor{deblue}{RGB}{11,132,147}
\definecolor{ocra}{RGB}{204, 119, 34}
\usepackage{enumitem}
\usepackage{tikz}

\nolinenumbers

\usepackage{lipsum}
\usepackage{amsfonts}
\usepackage{graphicx}
\usepackage{epstopdf}
\usepackage{algorithmic}
\ifpdf
  \DeclareGraphicsExtensions{.eps,.pdf,.png,.jpg}
\else
  \DeclareGraphicsExtensions{.eps}
\fi

\usepackage{enumitem}
\setlist[enumerate]{leftmargin=.5in}
\setlist[itemize]{leftmargin=.5in}

\newsiamremark{remark}{Remark}
\newsiamremark{hypothesis}{Hypothesis}
\crefname{hypothesis}{Hypothesis}{Hypotheses}
\newsiamthm{claim}{Claim}

\headers{Deep Block Proximal Linearised Minimisation Algorithm}{C. Huang, Z. Wu, Y. Cheng, T. Zeng, C. Sch\"onlieb, and A. Aviles-Rivero}

\title{Deep Block Proximal Linearised Minimisation Algorithm for Non-convex Inverse Problems\thanks{Submitted to the editors DATE.
\funding{ZW acknowledges support from the National Natural Science Foundation of China grant 12001286 and the China Postdoctoral Science Foundation grants 2022M711672. YC acknowledges funding from the Cambridge Centre for Data-Driven Discovery and Accelerate Programme for Scientific Discovery, made possible by a donation from Schmidt Futures.
TZ acknowledges support from the NSFC/RGC N\_CUHK 415/19, ITF ITS/173/22FP, RGC 14300219, 14302920, 14301121, and CUHK Direct Grant for Research. CBS acknowledges support from the Philip Leverhulme Prize, the Royal Society Wolfson Fellowship, the EPSRC advanced career fellowship EP/V029428/1, EPSRC grants EP/S026045/1 and EP/T003553/1, EP/N014588/1, EP/T017961/1, the Wellcome Innovator Awards 215733/Z/19/Z and 221633/Z/20/Z, CCMI and the Alan Turing Institute. AAR gratefully acknowledges funding from the Cambridge Centre for Data-Driven Discovery and Accelerate Programme for Scientific Discovery, made possible by a donation from Schmidt Futures, ESPRC Digital Core Capability Award, and CMIH and CCIMI, University of Cambridge.}}}

\author{Chaoyan Huang\thanks{Department of Mathematics, The Chinese University of Hong Kong, Hong Kong  
  (\email{cyhuang@math.cuhk.edu.hk}).}
    \and Zhongming Wu\thanks{School of Management Science and Engineering,
   Nanjing University of Information Science and Technology,
   Nanjing, China 
  (\email{wuzm@nuist.edu.cn}).}
    \and Yanqi Cheng\thanks{Department of Applied Mathematics and Theoretical Physics, 
   University of Cambridge 
   (\email{yc443@cam.ac.uk}).}
    \and Tieyong Zeng\thanks{Department of Mathematics, The Chinese University of Hong Kong, Hong Kong  (\email{zeng@math.cuhk.edu.hk}).}
    \and \qquad\qquad\qquad Carola-Bibiane Sch\"onlieb\thanks{Department of Applied Mathematics and Theoretical Physics, 
   University of Cambridge 
   (\email{cbs31@cam.ac.uk}).}
    \and Angelica I. Aviles-Rivero\thanks{Department of Applied Mathematics and Theoretical Physics, 
   University of Cambridge 
   (\email{ai323@cam.ac.uk}).}
}

\usepackage{amsopn}

\ifpdf
\hypersetup{
  pdftitle={Deep Block Proximal Linearised Minimisation Algorithm},
  pdfauthor={C. Huang, Z. Wu, Y. Cheng, T. Zeng, C. Sch\"onlieb, and A. Aviles-Rivero}
}
\fi

\begin{document}

\maketitle

\begin{abstract}
Image restoration is typically addressed through non-convex inverse problems, which are often solved using first-order block-wise splitting methods. In this paper, we consider a general type of non-convex optimisation model that captures many inverse image problems and present an inertial block proximal linearised minimisation (iBPLM) algorithm. Our new method unifies the Jacobi-type parallel and the Gauss-Seidel-type alternating update rules, and extends beyond these approaches. The inertial technique is also incorporated into each block-wise subproblem update, which can accelerate numerical convergence. Furthermore, we extend this framework with a plug-and-play variant (PnP-iBPLM) that integrates deep gradient denoisers, offering a flexible and robust solution for complex imaging tasks. We provide comprehensive theoretical analysis, demonstrating both subsequential and global convergence of the proposed algorithms. To validate our methods, we apply them to multi-block dictionary learning problems in image denoising and deblurring. Experimental results show that both iBPLM and PnP-iBPLM significantly enhance numerical performance and robustness in these applications.
\end{abstract}



\section{Introduction}

Image processing is crucial in solving a broad spectrum of inverse problems. 
Over the past decades, a large body of literature has been proposed to address imaging challenges. These range from classic models, such as the Rudin-Osher-Fatemi (ROF) model \cite{rudin1992nonlinear}, to modern approaches like dictionary learning~\cite{aharon2006k} and tight frame-based methods~\cite{cai2008framelet}. Among these, the dictionary learning paradigm stands out for its ability to approximate clean images ($I$) using dictionaries ($D$) and corresponding sparse coefficient vectors ($X$), effectively framed as  $DX = I$.  However, relying on the $\ell_0$ norm to maintain sparsity in $X$ makes this optimisation problem non-convex and involves multiple variables.

Different from other models, the dictionary learning approach aims to approximate the clean image $I$ using the dictionary $D$ and the corresponding sparse coefficient vector $X$. This leads to the following image denoising model:

\begin{equation}\label{dic}
    \min_{D,X}\frac{1}{2}\Vert DX-Y\Vert^2 + \lambda_D\phi_D(D)+\lambda_X\phi_X(X),
\end{equation}
where $Y$ is the input image, $\lambda_D$ and $\lambda_X$ are positive parameters, $\phi_D$ and $\phi_X$ are regularisers of $D$ and $X$, respectively. 
The optimal $D$ and $X$ are typically obtained via alternating-based \cite{elad2006image} or block-based \cite{xu2016fast} algorithms. The clean image is then reconstructed by $DX = I$. However, this model does not adapt well to other image restoration tasks. To address this, the Bayesian Maximum A Posteriori theory is incorporated to establish a more general model: 
\begin{equation}\label{map}
    \min_{D,X,I}\frac{1}{2}\Vert BI-Y\Vert^2+\frac{\eta}{2}\Vert DX-I\Vert^2 + \lambda_D\phi_D(D)+\lambda_X\phi_X(X),
\end{equation}
where $\eta$ is a positive parameter and $B$ is a linear operator for different image restoration tasks. 
For instance, when  $B$ is the identity operator, \eqref{map} functions as a denoising model; conversely, when $B$  is a sampling operator, the model serves as a super-resolution model.
Due to the sparsity of the dictionary learning model, detailed information can often   mistakenly treated as noise. To preserve edge details,  the works of that~\cite{huang2021quaternion} and~\cite{wu2022total} introduced the total variation prior into model \eqref{map}. Furthermore, various plug-and-play learning priors are utilised to replace the traditional knowledge-based regularisation, introducing inexplicitness and non-convexity into the model. Consequently, existing methods may not efficiently solve these complex models.

The aforementioned non-convex dictionary learning model motivates us to consider the following type of structural non-convex optimisation problem: 
\begin{equation}\label{objective}
\min~\left\{F\left(x_1, \ldots, x_p\right):=\sum_{i=1}^p \theta_i\left(x_i\right)+h\left(x_1, \ldots, x_p\right)~\Big|~ x_i \in \mathbb{R}^{n_i}\right\},
\end{equation}
where $h$ is a block-coordinate-wise Lipschitz smooth function and $\theta_i, i=1, \ldots, p$, are proper closed (possibly non-convex) functions. 
The Jacobi-type and the Gauss-Seidel-type block coordinate descent (BCD) methods are two categories of solving methods for \eqref{objective}.  Both methods deal with each block-variable $x_i$ individually in \eqref{objective}.
The Jacobi-type methods \cite{gan2024block,razaviyayn2014parallel,xu2017globally} update the variables in parallel, while the Gauss-Seidel-type algorithms \cite{attouch2010proximal,bolte2014proximal,phan2023inertial,le2020inertial} update them alternately, one by one. However, their effectiveness depends on the complexity of solving the block-wise subproblems.

Block-wise subproblems typically involve evaluating the proximal operator~\cite{parikh2014proximal} of certain functions, which may include total variational (TV) or learning-based regularisation in inverse image problems. However, most of these proximal subproblems lack a closed-form solution, particularly in cases of implicit learning regularisation. Fortunately, the plug-and-play~\cite{venkatakrishnan2013plug} (PnP) method can effectively address this challenge by alternately solving the subproblems using a trained deep neural network.

\textbf{Contributions.} To our knowledge, no existing work integrates the Gauss-Seidel-type BCD algorithms with the PnP approach. This paper aims to fill this gap. Additionally, this paper seeks to unify the Jacobi-type and the Gauss-Seidel-type BCD methods within a single algorithmic framework. Furthermore, inertial acceleration is introduced to enhance convergence speed, and the theoretical convergence of the unified algorithm, with or without the PnP prior, is examined. We highlight:
\begin{itemize}
    \item  We introduce a novel general inertial block proximal linearised minimisation (iBPLM) algorithm framework for solving the non-convex and non-smooth optimisation problem of that~\eqref{objective}. Our new method synthesises the Jacobi-type parallel and the Gauss-Seidel-type alternating update rules, thereby enhancing performance and properties.
\item Secondly, inspired by the principles of the Plug-and-Play (PnP) method, we redefine the iBPLM algorithm as the PnP-iBPLM. This version integrates gradient step-based deep priors. It significantly boosts the algorithm's ability to tackle more intricate and challenging optimisation problems. As a result, it sets a new standard in algorithm performance and adaptability.

\item By employing the Kurdyka–Łojasiewicz property, we have rigorously proved the subsequential and global convergence of both the iBPLM and PnP-iBPLM algorithms. This provides a solid theoretical underpinning, ensuring consistent performance across a range of complex optimisation scenarios. 

\item We extensively validate the theory with a range of numerical and visual results for image denoising and deblurring tasks. We demonstrate that our framework leads to better approximations than existing techniques.
\end{itemize} 

The remainder of this work is organized as follows. In section \ref{relatedwork}, we review some related works including the block coordinate descent optimisation algorithms, inertial acceleration, plug-and-play methods, and basic knowledge of dictionary learning-based methods in image restoration. After that, the proposed algorithms with and without deep denoiser are theoretically analyzed in section \ref{proposed}. In section \ref{results}, the imaging models with two blocks and four blocks are tested to illustrate the effectiveness of the proposed schemes. Final remarks are given in section \ref{conclusion}.

\section{Related Work}\label{relatedwork}

\subsection{BCD-based Optimisation Algorithms} 
Block coordinate descent (BCD) methods are commonly used for imaging problems by updating one block of variables at a time while fixing others. 
These methods can be broadly categorised into two types: alternating-based and block-based.
Alternating-based methods, such as the proximal alternating linearised minimisation (PALM) algorithm, update subsets of variables sequentially, effectively merging  the benefits of proximal operators with alternating minimisation. PALM is particularly effective for non-smooth and non-convex problems,  and it has demonstrated global convergence under specified conditions \cite{bolte2014proximal}. 
Recent studies have further validated its effectiveness and convergence properties in various settings \cite{wang2023linear, wang2024stochastic}.
There are three principal types of block-based methods: classical block-based \cite{tseng2001convergence}, proximal block-based \cite{razaviyayn2013unified}, and proximal gradient block-based \cite{bolte2014proximal} methods. The classical block-based method alternates the minimisation of block functions within the objective but may be inadequate for non-convex problems. The proximal block-based method enhances this approach by integrating block functions with a proximal term, thereby ensuring global convergence under specific conditions \cite{attouch2010proximal}. The proximal gradient block-based method, on the other hand, minimises a proximal linearisation of the objective and achieves global convergence when block functions are Lipschitz smooth \cite{bolte2014proximal}. Moreover, when block functions exhibit relative smoothness, further studies have demonstrated global convergence \cite{bauschke2017descent, lu2018relatively, ahookhosh2021multi, hien2021algorithms}.

\subsection{Inertial Acceleration} 
Numerous studies have focused on enhancing the convergence rate of gradient-based first-order methods, as demonstrated in~\cite{nesterov2013gradient,pock2016inertial,wu2024extrapolated}.
A widely adopted strategy involves incorporating an inertial force, often referred to as extrapolation, into the iterative scheme. This approach leverages the outcomes of the previous two iterations to update the next iterate, ultimately resulting in the development of inertial accelerated methods.
Well-known techniques such as the classic heavy-ball method and Nesterov acceleration are exemplary of this type of accelerated method.
In convex settings, extensive research has shown that the convergence rate can be theoretically accelerated by selecting optimal extrapolated schemes~\cite{attouch2016rate,attouch2018fast,beck2009fast}. Over the past decade, inertial acceleration has been adapted for non-convex optimisation~\cite{ahookhosh2021block,pock2016inertial,ochs2014ipiano,ochs2019unifying,phan2023inertial,qu2024partially,wu2019general}, and has been effectively integrated into BCD-based optimisation algorithms~\cite{xu2013block,xu2017globally,razaviyayn2013unified,le2020inertial,phan2023inertial}. The convergence of these inertial methods to a stationary point can be guaranteed under the Kurdyka–{\L}ojasiewicz framework for non-convex problems \cite{attouch2010proximal}. Although the inertial effect may not theoretically accelerate the convergence rate in non-convex settings, numerous practical implementations have demonstrated its effectiveness.

\subsection{Plug-and-Play (PnP) Algorithms} 
PnP methods integrate denoising priors into splitting algorithms, which are widely used for imaging problems~\cite{zhang2017learning}. PnP-ADMM, introduced by Venkatakrishnan et al. \cite{venkatakrishnan2013plug}, replaces the proximal subproblem with a denoising prior, offering a flexible framework for image restoration. Subsequent approaches like PnP-FBS \cite{wu2020deep} and PnP-DRS \cite{buzzard2018plug} have demonstrated empirical success across diverse applications. However, theoretical guarantees are limited, often relying on assumptions such as denoiser averaging or nonexpansiveness~\cite{sun2019online, sun2021scalable}. A significant challenge is ensuring nonexpansiveness in deep denoisers, which is crucial for convergence. Off-the-shelf deep denoisers often lack 1-Lipschitz continuity, which affects performance~\cite{zhang2021plug}. Ryu et al. \cite{ryu2019plug} proposed normalising each layer using its spectral norm, but this limits residual skip connections. The work of that~\cite{hurault2022gradient} addressed this by training a denoiser with a gradient-based PnP prior. Recent advancements have led to exploring convergence guarantees and applying combined PnP methods with extrapolated DYS algorithms \cite{wu2024extrapolated}. These developments potentially enhance both convergence and applicability in complex inverse problems.

\subsection{Dictionary Learning-based Image Restoration} Dictionary learning is a powerful technique for image restoration. The goal is to learn a dictionary from given data so that each image patch can be sparsely represented by a linear combination of dictionary atoms. This method is especially effective for inverse problems like image denoising and deblurring. The seminal work by Aharon et al.~\cite{aharon2006k} introduced the K-SVD algorithm. It iteratively refines both the dictionary and the sparse representations, leading to significant improvements in image quality. Subsequent research has expanded on this foundation, exploring various aspects of dictionary learning and sparse coding~\cite{cai2008framelet, elad2006image, mairal2009online}. Moreover, recent advancements focus on integrating dictionary learning with deep learning techniques. This integration involves learning the dictionary from neural networks~\cite{papyan2017convolutional, scetbon2021deep} or combining it with the PnP method to handle the dictionary learning model~\cite{sujithra2022compressed, yang2022revisit}. However, most of these studies overlook the theoretical analysis of the dictionary-based model.

The aforementioned methods are primarily encompassed within the framework of problem~\eqref{objective}. We detail the most relevant works along with their problem settings, algorithm designs, and theoretical analyses in Table~\ref{tab: pro}. This summary clearly demonstrates that our proposed iBPLM and PnP-iBPLM algorithms comprehensively cover all these aspects, effectively addressing the limitations identified in previous studies and providing a robust, theoretically sound framework for optimisation.

\begin{table}[t!]
    \centering\small
        \caption{\textbf{Comparative Overview of Existing Techniques}. This table delineates differences in problem settings, algorithm designs, and convergence results. Abbreviations include: Non. (Non-convex), Para. (Parallel), Alter. (Alternating), Iner. (Inertial), PnP (Plug-and-Play), Sub. (Subsequential), and Global (Global convergence).}
    \label{tab: pro}
\resizebox{\hsize}{!}
{
    \begin{tabular}{c|cccccccc}
    \hline
     \multirow{2}{*}{\textsc{Existing Techniques}} & 
    \multicolumn{2}{c}{\cellcolor[HTML]{EFEFEF}Problem setting} & \multicolumn{4}{|c}{\cellcolor[HTML]{EFEFEF}Algorithm design} &\multicolumn{2}{|c}{\cellcolor[HTML]{EFEFEF}Theoretical results}\\ \cline{2-9}
                  &\cellcolor[HTML]{F5F5F5}
                  $h$ Non. & 
                  \cellcolor[HTML]{F5F5F5}$\theta_i$ Non. & \cellcolor[HTML]{F5F5F5}Para. & \cellcolor[HTML]{F5F5F5}Alter. &\cellcolor[HTML]{F5F5F5} Iner. &\cellcolor[HTML]{F5F5F5}PnP &\cellcolor[HTML]{F5F5F5} Sub. &\cellcolor[HTML]{F5F5F5} Global\\ \hline
    (Xu \& Yin, 2013) \cite{xu2013block}
                   &\ding{51}&\textcolor{red}{\ding{55}}&\textcolor{red}{\ding{55}}&\ding{51}&\ding{51}&\textcolor{red}{\ding{55}}&\ding{51}&\ding{51}\\ 
    (Razaviyayn et al., 2014) \cite{razaviyayn2014parallel}
                    &\ding{51} & \textcolor{red}{\ding{55}}& \ding{51} &\textcolor{red}{\ding{55}}&\ding{51}&\textcolor{red}{\ding{55}}&\ding{51}& \textcolor{red}{\ding{55}}\\ 
    (Hong et al., 2017) \cite{hong2017iteration}
                  &\textcolor{red}{\ding{55}}&\textcolor{red}{\ding{55}}&\textcolor{red}{\ding{55}}&\ding{51}&\ding{51}&\textcolor{red}{\ding{55}}&\ding{51}&\ding{51}\\ 
    (Xu \& Yin, 2017) \cite{xu2017globally}
     &\ding{51}&\ding{51}&\ding{51}&\textcolor{red}{\ding{55}}&\ding{51}&\textcolor{red}{\ding{55}}&\ding{51}&\ding{51}\\ 
     (Teboulle \& Vaisbourd, 2020) \cite{teboulle2020novel}
      &\ding{51}&\ding{51}&\ding{51}&\textcolor{red}{\ding{55}}&\textcolor{red}{\ding{55}}&\textcolor{red}{\ding{55}}&\ding{51}&\ding{51}\\ 
    (Hien et al., 2020) \cite{le2020inertial}
        &\ding{51}& \ding{51}& \textcolor{red}{\ding{55}}&\ding{51}&\ding{51}&\textcolor{red}{\ding{55}}&\ding{51}&\ding{51}\\  
    (Yang et al., 2020) \cite{yang2020inexact}
    &\ding{51}&\textcolor{red}{\ding{55}} &\textcolor{red}{\ding{55}} & \ding{51}&\textcolor{red}{\ding{55}}&\textcolor{red}{\ding{55}}&\ding{51}&\ding{51}\\ 
    (Sujithra \& Sugitha, 2022) \cite{sujithra2022compressed}
    &\ding{51}&\ding{51} &\textcolor{red}{\ding{55}} & \textcolor{red}{\ding{55}}&\textcolor{red}{\ding{55}}&\ding{51}& \textcolor{red}{\ding{55}}& \textcolor{red}{\ding{55}}\\ 
    (Phan et al., 2023) \cite{phan2023inertial}
     &\ding{51} &\ding{51} &\textcolor{red}{\ding{55}}& \ding{51}& \ding{51}&\textcolor{red}{\ding{55}}&\ding{51}&\ding{51}\\ 
    (Gan et al., 2024) \cite{gan2024block}
     & \ding{51} & \ding{51} &\ding{51} &\textcolor{red}{\ding{55}}&\textcolor{red}{\ding{55}}&\ding{51}& \textcolor{red}{\ding{55}}& \textcolor{red}{\ding{55}}\\ 
     \textcolor{blue}{\FiveStar} Ours
       & \textcolor{blue}{\ding{51}} & \textcolor{blue}{\ding{51}} & \textcolor{blue}{\ding{51}} & \textcolor{blue}{\ding{51}} & \textcolor{blue}{\ding{51}} & \textcolor{blue}{\ding{51}} &\textcolor{blue}{\ding{51}} & \textcolor{blue}{\ding{51}}\\ \hline         %
    \end{tabular}
    }  
\end{table}

\section{Proposed Method}\label{proposed}
This section details our proposed unified inertial block proximal linearised minimisation method and its Plug-and-Play (PnP) variant, along with establishing their theoretical convergence.

{\bf Notation.} Denote $x_{<i}:=(x_1, x_2, \ldots, x_{i-1})$, $x_{>i}:=(x_{i+1}, x_{i+2}, \ldots, x_p)$, and $x_{\neq i}:=(x_{<i}, x_{>i})$. The distance between a point $x \in \mathbb{R}^n$ and a closed and convex set $\mathcal{G} \subseteq \mathbb{R}^n$ is defined as $\operatorname{dist}(x, \mathcal{G}) := \min_{y \in \mathcal{G}} \|x-y\|$.
The indicator function for $\mathcal{G}$ assigns 0 to all $x \in \mathcal{G}$ and $+\infty$ otherwise.  
For $\gamma > 0$, the proximal operator of the function $f$ is defined by 
$
{\rm Prox}_{\gamma f}(x) :=  \arg\min_{y \in \mathbb{R}^n} \{f(y) + \frac{1}{2 \gamma} \|y-x\|^2\}. 
$

\subsection{Inertial block proximal linearised minimisation algorithm}
We first propose an inertial block proximal linearised minimisation (iBPLM) algorithm for solving the non-convex and non-smooth model of that~\eqref{objective}, referred to as Algorithm \ref{algo1}. This algorithm can be utilised to handle dictionary learning models with explicit regularisation.

\begin{algorithm}[h]
\caption{Inertial block proximal linearised minimisation algorithm (iBPLM)}\label{algo1}
\begin{algorithmic}
\STATE{{\bf Initialisation:} Select $w_{ij}\in[0,1]$ for $j<i$ and  $w_{ij}=1$ for $j\geq i$. Given $x_i^0$, set $x_i^{-1}=x_i^{0}$.
}
\FOR{$k=1, 2, 3, \ldots$}
\STATE{Update $x_j^{k,i}= (1-w_{ij})x_j^{k+1}+w_{ij}x_j^k,~~ \forall i,j=1,2,\ldots,p$.}
\FOR{$i=1, 2, 3, \ldots, p$}
\STATE{Choose $\alpha_i^k\in[0,1)$.}
\STATE{Update $\hat{x}_i^k=x_i^k+\alpha_i^k(x_i^k-x_i^{k-1})$;}
\STATE{Update $x_i^{k+1}={\rm Prox}_{\gamma_i^k\theta_i}(\hat{x}_i^{k}-\gamma_i^k \nabla_i h({\hat x}_i^k, x_{\neq i}^{k,i}))$.}
\ENDFOR
\IF{stopping criterion is satisfied,}
\STATE{
Return $(x_1^k, x_2^k, \ldots, x_p^k)$.
}
\ENDIF
\ENDFOR
\end{algorithmic}
\end{algorithm}

By introducing a weight matrix \(W := (w_{ij})_{p \times p}\), Algorithm \ref{algo1} establishes a unified framework. This framework incorporates the Jacobi-type parallel update rule among task blocks, as well as the Gauss-Seidel-type alternating update within each block. Specifically, the matrix \(W\) is employed to compute \(x_j^{k,i}\), defined as \(x_j^{k,i} = (1-w_{ij})x_j^{k+1} + w_{ij}x_j^k\) for all \(i, j = 1, 2, \ldots, p\), where \(w_{ij} \in [0,1]\) for \(j < i\) and \(w_{ij} = 1\) for \(j \geq i\). When \(w_{ij} = 1\) and \(w_{ij} = 0\) $\forall$ \(j < i\), the matrix $W=W_1$ and $W=W_2$ defined by:
\begin{equation}\label{W_define}
  W_1=\left[\begin{array}{cccc}
1 & 1 & \cdots & 1 \\
1 & 1 & \cdots & 1 \\
\vdots & \ddots & \ddots & \vdots \\
1 & \cdots & 1 & 1
\end{array}\right] \quad \text { and } \quad W_2=\left[\begin{array}{cccc}
1 & 1 & \cdots & 1 \\
0 & 1 & \cdots & 1 \\
\vdots & \ddots & \ddots & \vdots \\
0 & \cdots & 0 & 1
\end{array}\right],    
\end{equation}
 
Algorithm \ref{algo1} simplifies to the Jacobi-type and Gauss-Seidel-type BPLM, respectively. Additionally, to accelerate the convergence of the unified algorithm, an inertial step is incorporated into the updates of the block-subproblems.

We review the definitions of subdifferential and Kurdyka-{\L}ojasiewicz (KL) property for further analysis.
\begin{definition} \label{def2.1}{\rm\cite{attouch2013convergence, bolte2014alternating}} \rm(Subdifferentials)
Let $f:\mathbb {R}^{n}\rightarrow (-\infty,+\infty]$ be a proper and lower semicontinuous function.
\begin{itemize}
\item[(i)] For a given ${ x}\in {\rm dom} f$, the Fr\'{e}chet subdifferential of $f$ at ${ x}$, written by $\hat{\partial}f({ x})$, is the set of all vectors ${ u}\in \mathbb{R}^n$ satisfying\vspace{-0.05in}
$$\liminf_{{ y}\neq { x}, { y}\rightarrow { x}}\frac{f({ y})-f({ x})-\langle { u},{ y}-{ x}\rangle}{\|{ y}-{ x}\|}\geq0,\vspace{-0.05in}$$
and we set $\hat{\partial}f({ x}) = \emptyset$ when ${ x}\notin {\rm dom}f$.
\vspace{0.2cm}
\item[(ii)] The limiting-subdifferential, or simply the subdifferential, of $f$ at ${ x}$, written by $\partial f({ x})$, is defined by 
\begin{equation}\label{pf}
\partial f({ x}):=\{{ u}\in\mathbb{R}^n\; | \; \exists ~ { x}^k\rightarrow { x}, ~{\rm s.t.}~f({ x}^k)\rightarrow f({ x})
 ~{\rm and}~ \hat{\partial}f({ x}^k) \ni { u}^k \rightarrow { u} \}.  \end{equation}

\item[(iii)] A point ${ x}^*$ is called (limiting-)critical point or stationary point of $f$ if it satisfies $0\in\partial f({ x}^*)$, and the set of critical points of $f$ is denoted by ${\rm crit} f$.
\end{itemize}
\end{definition}

Next, we recall the KL property \cite{attouch2010proximal, bolte2014alternating}, which is important in the convergence analysis.
\begin{definition}\label{def2.2}(KL property and KL function)
Let $f:\mathbb{R}^{n}\rightarrow (-\infty,+\infty]$ be a proper and lower semicontinuous function.
\begin{itemize}
\item[$(a)$] The function $f$ is said to have KL property at ${ x}^*\in{\rm dom}(\partial f)$ if there exist $\eta\in(0,+\infty]$, a neighborhood $U$ of ${ x}^*$ and a continuous and concave function $\varphi:[0,\eta)\rightarrow \mathbb{R}_+$ such that
\begin{itemize}
\item[\rm(i)] $\varphi(0)=0$ and $\varphi$ is continuously differentiable on $(0,\eta)$ with $\varphi'>0$;

\item[\rm(ii)] for all ${ x}\in U\cap\{{ z} \in \mathbb{R}^n\; | \; f({ x}^*) < f({ z}) < f({ x}^*)+\eta\}$, the following KL inequality holds: 
\begin{equation} \label{kl}
\varphi'(f({ x})-f({ x}^*)){\rm dist}(0,\partial f({ x}))\geq1. 
\end{equation}
\end{itemize}

\item[$(b)$] If $f$ satisfies the KL property at each point of dom$(\partial f)$, then $f$ is called a KL function.
\end{itemize}
\end{definition}

\begin{remark}
KL functions exhibit remarkable versatility and are extensively applied in various domains, including semi-algebraic analysis, subanalytic analysis, and log-exp functions.
Concrete examples of KL functions can be found in \cite{attouch2010proximal, attouch2013convergence, bolte2014alternating}. These examples encompass many common instances such as $\ell_p$-norm (where $p\geq 0$), indicator functions of semi-algebraic sets, and a majority of convex functions.
\end{remark}

Now we begin to analyse the theoretical convergence of iBPLM. To do so, we need the following mild assumption.

\begin{assumption}
The gradient of $h$ is block-coordinate-wise Lipschitz continuous; that is, for a given $x_{\neq i}$, it holds that
$$\|\nabla_ih(x_i,x_{\neq i})-\nabla_ih(y_i,x_{\neq i})\|\leq L_i(x_{\neq i})\|x_i-y_i\|,  \text{ for }i=1,2,\ldots,p.$$
\end{assumption}

Denote $L_i^k:=L_i(x_{<i}^{k+1},x_{>i}^k)$, and $L^k:=\max\{L_i^k, i=1,2,\ldots,p\}$. We summarise the property of $F$ in the following proposition. 
\begin{proposition}\label{NSDP}
For the sequence $\{x^k\}$ generated by the proposed iBPLM algorithm, it must satisfy
\begin{equation} 
\begin{aligned}
  & F(x^k) +\sum_{i=1}^{p}\frac{\xi_i^k}{2}\|x_i^{k}-x_i^{k-1}\|^2\geq F(x^{k+1})+  \sum_{i=1}^{p}\frac{\delta_i^k}{2}\Vert x_i^{k+1}- x_i^k\Vert^2, k=1,2,\ldots
\end{aligned}
\end{equation}
where $\xi_i^k:=\frac{\alpha_i^k\gamma_i^kL_i^k+\alpha_i^k}{\gamma_i^k}$, $\delta_i^k:=\frac{1-\alpha_i^k -\gamma_i^k L_i^k-\alpha_i^k\gamma_i^kL_i^k-\gamma_i^kw'_iL^k}{\gamma_i^k}$, and $w'_i=\sum_{q=i+1}^{p}w_{qi}$.
\end{proposition}

\begin{proof}
It follows from Algorithm \ref{algo1} that
$$x_i^{k+1}=\underset{x_i}{\arg\min}~ \left\{\langle x_i-\hat{x}_i^k, \nabla_ih({\hat x}_i^k, x_{\neq i}^{k,i})\rangle+\frac{1}{2\gamma_i^k}\Vert x_i-\hat{x}_i^k\Vert^2 + \theta_i(x_i)\right\},$$
which implies that
\begin{equation}\label{minopt} 
\begin{aligned}
  & \langle x_i^k-\hat{x}_i^k, \nabla_ih({\hat x}_i^k, x_{\neq i}^{k,i})\rangle+\frac{1}{2\gamma_i^k}\Vert x_i^k-\hat{x}_i^k\Vert^2 + \theta_i(x_i^k)\\
   & \geq \langle x_i^{k+1}-\hat{x}_i^k, \nabla_ih({\hat x}_i^k, x_{\neq i}^{k,i})\rangle+\frac{1}{2\gamma_i^k}\Vert x_i^{k+1}-\hat{x}_i^k\Vert^2 + \theta_i(x_i^{k+1}).
\end{aligned}
\end{equation}
Hence,
\begin{equation} \label{add1}
\begin{aligned}
  & \langle x_i^k-x_i^{k+1}, \nabla_ih({\hat x}_i^k, x_{\neq i}^{k,i})\rangle+\frac{1}{2\gamma_i^k}\Vert x_i^k-\hat{x}_i^k\Vert^2 + \theta_i(x_i^k)\\
   & \geq  \frac{1}{2\gamma_i^k}\Vert x_i^{k+1}-\hat{x}_i^k\Vert^2 + \theta_i(x_i^{k+1}).
\end{aligned}
\end{equation}
Since $\nabla_i h$ is Lipschitz continuous, we have
\begin{equation}\label{add2}
 h({x}_{<i}^{k+1},{x}_i^{k+1}, x_{>i}^k)-h({x}_{<i}^{k+1},{x}_i^{k}, x_{>i}^k)-\langle x_i^{k+1}-x_i^k,\nabla_i h({x}_{<i}^{k+1},{x}_i^{k}, x_{>i}^k)\rangle \leq \frac{L_i^k}{2}\|x_i^{k+1}-x_i^k\|^2. 
\end{equation}
Combining \eqref{add1} and \eqref{add2}, and recalling the definition of $F$ in \eqref{objective}, we have
\begin{equation} \label{add3}
\begin{aligned}
  & F({x}_{<i}^{k+1},{x}_i^{k}, x_{>i}^k)+\langle x_i^{k+1}-x_i^{k}, \nabla_i h({x}_{<i}^{k+1},{x}_i^{k}, x_{>i}^k)-\nabla_ih({\hat x}_i^k, x_{\neq i}^{k,i})\rangle  \\
   & \geq F({x}_{<i}^{k+1},{x}_i^{k+1}, x_{>i}^k)+  \frac{1}{2\gamma_i^k}\Vert x_i^{k+1}-\hat{x}_i^k\Vert^2 -\frac{1}{2\gamma_i^k}\Vert x_i^k-\hat{x}_i^k\Vert^2-\frac{L_i^k}{2}\|x_i^{k+1}-x_i^k\|^2\\
   &=F({x}_{<i}^{k+1},{x}_i^{k+1}, x_{>i}^k)+  \frac{1}{2\gamma_i^k}\Vert x_i^{k+1}- {x}_i^k\Vert^2 -\frac{\alpha_i^k}{\gamma_i^k}\langle x_i^{k+1}-x_i^k,x_i^k-x_i^{k-1} \rangle-\frac{L_i^k}{2}\|x_i^{k+1}-x_i^k\|^2\\
   &\geq F({x}_{<i}^{k+1},{x}_i^{k+1}, x_{>i}^k)+  \frac{1-\alpha_i^k-\gamma_i^k L_i^k}{2\gamma_i^k}\Vert x_i^{k+1}- {x}_i^k\Vert^2 -\frac{\alpha_i^k}{2\gamma_i^k}\|x_i^k-x_i^{k-1} \|^2,
\end{aligned}
\end{equation}
Note that 
\begin{equation} \label{add4}
\begin{aligned}
  & \langle x_i^{k+1}-x_i^{k}, \nabla_i h({x}_{<i}^{k+1},{x}_i^{k}, x_{>i}^k)-\nabla_ih({\hat x}_i^k, x_{\neq i}^{k,i})\rangle\\  
  &  = \langle x_i^{k+1}-x_i^{k}, \nabla_i h({x}_{<i}^{k+1},{x}_i^{k}, x_{>i}^k)-\nabla_ih(x^{k,i})\rangle+ 
  \langle x_i^{k+1}-x_i^{k}, \nabla_i h(x^{k,i})-\nabla_ih({\hat x}_i^k, x_{\neq i}^{k,i})\rangle\\
  &\leq   \|x_i^{k+1}-x_i^{k}\|\|\nabla_i h({x}_{<i}^{k+1},{x}_i^{k}, x_{>i}^k)-\nabla_ih(x^{k,i})\|+\frac{\alpha_i^kL_i^k}{2} \|x_i^{k+1}-x_i^{k}\|^2+\frac{\alpha_i^kL_i^k}{2}\|x_i^{k}-x_i^{k-1}\|^2\\
   &\leq   \|x_i^{k+1}-x_i^{k}\|\left(\sum_{j=1}^{i-1}w_{ij}L^k \|x_j^{k+1}-x_j^{k}\|\right)+\frac{\alpha_i^kL_i^k}{2} \|x_i^{k+1}-x_i^{k}\|^2+\frac{\alpha_i^kL_i^k}{2}\|x_i^{k}-x_i^{k-1}\|^2,
\end{aligned}
\end{equation}
where $L^k=\max_{i}\{L_i^k,i=1,2,\ldots,p\}$. Since  $F({x}^{k+1})-F(x^k)=\sum_{i=1}^{p}(F({x}_{<i}^{k+1},{x}_i^{k+1}, x_{>i}^k)-F({x}_{<i}^{k+1},{x}_i^{k}, x_{>i}^k))$ with $x^{k+1} = \{x_i^{k+1}, \cdots, x_p^{k+1}\}$ and $x^{k} = \{x_i^{k}, \cdots, x_p^{k}\}$
according to the definition of $F$ in \eqref{objective}, it follows from \eqref{add3} and \eqref{add4} that
\begin{equation} 
\begin{aligned}
  & F(x^k) +\sum_{i=1}^{p}\frac{\alpha_i^k\gamma_i^kL_i^k+\alpha_i^k}{2\gamma_i^k}\|x_i^{k}-x_i^{k-1}\|^2\\
   &\geq F(x^{k+1})+  \sum_{i=1}^{p}\frac{1-\alpha_i^k -\gamma_i^k L_i^k-\alpha_i^k\gamma_i^kL_i^k}{2\gamma_i^k}\Vert x_i^{k+1}- x_i^k\Vert^2\\
   &\quad-\sum_{i=1}^{p}\|x_i^{k+1}-x_i^{k}\|\left(\sum_{j=1}^{i-1}w_{ij}L^k \|x_j^{k+1}-x_j^{k}\|\right)\\
    &\geq F(x^{k+1})+  \sum_{i=1}^{p}\frac{1-\alpha_i^k -\gamma_i^k L_i^k-\alpha_i^k\gamma_i^kL_i^k}{2\gamma_i^k}\Vert x_i^{k+1}- x_i^k\Vert^2
   -\frac{w'_iL^k}{2}\sum_{i=1}^{p}\|x_i^{k+1}-x_i^{k}\|^2\\
   &= F(x^{k+1})+  \sum_{i=1}^{p}\frac{1-\alpha_i^k -\gamma_i^k L_i^k-\alpha_i^k\gamma_i^kL_i^k-\gamma_i^kw'_iL^k}{2\gamma_i^k}\Vert x_i^{k+1}- x_i^k\Vert^2,
\end{aligned}
\end{equation}
where $w'_i=\sum_{q=i+1}^{p}w_{qi}$.
This completes the proof.
\end{proof}

Note that if the sequences \(\{\xi_i^k\}\) and \(\{\delta_i^k\}\) are positive and adhere to certain relationships, the variant of the objective function in~\eqref{objective} is non-increasing with respect to the iteration number \(k\). The following lemma outlines the parameter conditions. 

\begin{lemma}\label{bound}
    Let $\{x^k\}$ be the sequence generated by the iBPLM algorithm.   If  $\xi_i^{k+1}\leq C \delta_i^k$ for some constant $C\in(0,1)$, and there exists a positive parameter $\underline{l}$ such that $\min_{i,k}\{\frac{\delta_i^k}{2}\}\geq\underline{l}$ then we have
        \begin{equation}
            \sum_{k=0}^{+\infty} \sum_{i=1}^m\left\|x_i^{k+1}-x_i^k\right\|^2<+\infty.
        \end{equation}
\end{lemma}

\begin{proof}
    (i) It follows from Proposition \ref{NSDP} that
\begin{equation}
F(x^k) +\sum_{i=1}^{p}\frac{\xi_i^k}{2}\|x_i^{k}-x_i^{k-1}\|^2\geq F(x^{k+1})+  \sum_{i=1}^{p}\frac{\delta_i^k}{2}\Vert x_i^{k+1}- x_i^k\Vert^2, k=1,2,\ldots,
\end{equation}
where $\xi_i^k$ and $\delta_i^k$ are positive parameters and $\xi_i^{k+1}\leq C \delta_i^k$ for some constant $C\in(0,1)$. Hence, 
we have
\begin{equation}
\label{sumF}
    F(x^{k+1})+
    \sum_{i=1}^p\frac{\delta_i^k}{2}\Vert x_i^{k+1}-x_i^k\Vert^2
    \leq
    F(x^k)+\sum_{i=1}^pC\frac{\delta_i^{k-1}}{2}\left\|x_i^k-x_i^{k-1}\right\|^2.
\end{equation}
Summing up $k=0$ to $K-1$, we get
\begin{equation}
\begin{aligned}
    &F(x^K)+\sum_{i=1}^p\frac{\delta_i^{K-1}}{2}\Vert x_i^{K}-x_i^{K-1}\Vert^2+(1-C)\sum_{k=0}^{K-1}\sum_{i=1}^p\frac{\delta_i^k}{2}\Vert x_i^{k+1}-x_i^k\Vert^2\\
    \leq&
    F(x^0)+\sum_{i=1}^pC\frac{\delta_i^{-1}}{2}\left\|x_i^0-x_i^{-1}\right\|^2.
\end{aligned}
\end{equation}
This completes the proof. 
\end{proof}

\begin{remark}
Indeed, it is not difficult to guarantee the parameter condition $\xi_i^{k+1}\leq C \delta_i^k$.
For instance, if $W=W_2$ and the parameters are constant such that $\alpha_i^k=\alpha_i$ and $\gamma_i^k=\gamma_i$ for all $k$, it holds that $w'_i=0$, and the condition  $\xi_i^{k+1}\leq C \delta_i^k$ can be satisfied if
\begin{equation}
    0<\gamma_i<\frac{1}{L_i},\quad {\rm and }\quad 0\leq\alpha_i<\frac{1-\gamma_i L_i}{2+2\gamma L_i}.
\end{equation}
Note that in the implementation, as discussed in~\cite{li2016douglas,yang2017alternating}, one can initialise the algorithm with larger values of \(\alpha_i\) and \(\gamma_i\). If these parameters do not meet the required conditions, they should be decreased by a constant ratio. This adjustment is necessary if the sequence generated by the algorithm becomes unbounded, or if the successive changes in the sequence do not diminish sufficiently fast.
\end{remark}

Our first theoretical result concerns the subsequential convergence of the iBPLM to a stationary point of~\eqref{objective}. This result is detailed in the following theorem. 

\begin{theorem}\label{sub}
    Suppose that Proposition \ref{NSDP} and Lemma \ref{bound} hold. 
    Assume that the sequence $\left\{x^k\right\}$ generated by the proposed iBPLM Algorithm is bounded. Every limit point $x^*$ of $\left\{x^k\right\}$ is a critical point of problem \eqref{objective}.
\end{theorem}

\begin{proof}
Suppose a subsequence $\left\{x^{k_n}\right\}$ of $\left\{x^k\right\}$ converges to $x^* \in \mathcal{X}$. Lemma \ref{bound} implies that $x^{k_n-1} \rightarrow x^*$ and $x^{k_n+1} \rightarrow x^*$. Choosing $x_i=x_i^*$ and $k=k_n$ in \eqref{minopt}, we obtain
\begin{equation} 
\begin{aligned}
  & \langle x_i^*-x_i^{k_n+1}, \nabla_ih({\hat x}_i^{k_n}, x_{\neq i}^{k_n,i})\rangle+\frac{1}{2\gamma_i^{k_n}}\Vert x_i^*-\hat{x}_i^{k_n}\Vert^2 + \theta_i(x_i^*)\\
   & \geq  \frac{1}{2\gamma_i^{k_n}}\Vert x_i^{k_n+1}-\hat{x}_i^{k_n}\Vert^2 + \theta_i(x_i^{k_n+1}).
\end{aligned}
\end{equation}
Since $\nabla_i h$ is Lipschitz continuous, we have
\begin{equation}
 h({x}_{<i}^{k_n+1},{x}_i^{k_n+1}, x_{>i}^{k_n})-h({x}_{<i}^{k_n+1},{x}_i^{*}, x_{>i}^{k_n})-\langle x_i^{k_n+1}-x_i^*,\nabla_i h({x}_{<i}^{k_n+1},{x}_i^{*}, x_{>i}^{k_n})\rangle \leq \frac{L_i^{k_n}}{2}\|x_i^{k_n+1}-x_i^*\|^2. 
\end{equation}
From the definition of $F$, we have 
\begin{equation} 
\limsup_{n\rightarrow \infty}\frac{1}{2\gamma_i^{k_n}}\Vert x_i^{k_n+1}-\hat{x}_i^{k_n}\Vert^2 +h({x}_{<i}^{k_n+1},{x}_i^{k_n+1}, x_{>i}^{k_n})+ \theta_i(x_i^{k_n+1})\leq F(x^*).
\end{equation}
From the low semi-continuous and for $x_i$, we have 
\begin{equation}
    F(x^*)\leq h({x}_{<i}^{*},{x}_i, x_{>i}^{*}) + \theta_i(x_i),
\end{equation}
which means that $x_i^*$ is the minimiser of the problem 
\begin{equation}
   \min_{x_i} h({x}_{<i}^{*},{x}_i, x_{>i}^{*}) + \theta_i(x_i).
\end{equation}
Then from the optimality condition, we complete the proof.
\end{proof}

We demonstrate the global convergence of iBPLM if the objection function in~\eqref{objective} satisfies the Kurdyka–{\L}ojasiewicz property in the following theorem. 
\begin{theorem}\label{conv}
    Suppose that Proposition \ref{NSDP}, Lemma \ref{bound}, and assumption \ref{par} hold. 
    Assume that the  sequence $\left\{x^k\right\}$ generated by the proposed iBPLM Algorithm is bounded. Let $h$ be a continuously differentiable function, $F$ is a KL function and together with the existence of $\underline{l}$ in Lemma \ref{bound}, we also assume there exists $\bar{l}>0$, such that $\max_{i,k}\{\frac{\delta_i^k}{2}\}\leq\bar{l}$. For $C$ in Proposition \ref{NSDP} satisfying $C<\underline{l}/\bar{l}$, the whole generated sequence $\{x^k\}$ of the proposed iBPLM algorithm is convergent. 
\end{theorem}

\begin{assumption}\label{par}(\cite{phan2023inertial}, Assumption 3)
Let $u_i(x,z)=h(z)+\frac{1}{2\gamma_i^k}\Vert x-z\Vert^2+\langle x-z, \nabla_ih(z) \rangle$,
for any bounded subset of $\mathcal{X}$ and any $x, z$ in this subset, for ${s}_i \in \partial_{x_i} u_i(x, z)$, there exists ${t}_i \in \partial_{x_i} f(x)$ such that
$$
\left\|{s}_i-{t}_i\right\| \leq B_i\|x-z\|
$$
for some constant $B_i$ that may depend on the subset.
\end{assumption}
\begin{lemma}\label{lemma}(\cite{le2020inertial}, Theorem 2) Let $\Phi: \mathbb{R}^N \rightarrow(-\infty,+\infty]$ be a proper and lower semicontinuous function which is bounded from below. Let $\mathcal{A}$ be a generic algorithm which generates a bounded sequence $\left\{z^k\right\}$ by $z^0 \in \mathbb{R}^N, z^{k+1} \in \mathcal{A}\left(z^k\right), k=0,1, \ldots$ Assume that there exist positive constants $\rho_1, \rho_2$ and $\rho_3$ and a non-negative sequence $\left\{\varphi_k\right\}_{k \in \mathbb{N}}$ such that the following conditions are satisfied:
\begin{itemize}
\item[(i)] Sufficient decrease property:
$$
\rho_1\left\|z^k-z^{k+1}\right\|^2 \leq \rho_2 \varphi_k^2 \leq \Phi\left(z^k\right)-\Phi\left(z^{k+1}\right), k=0,1, \ldots
$$ 
\item[(ii)] Boundedness of subgradient:
$$
\left\|\omega^{k+1}\right\| \leq \rho_3 \varphi_k, \omega^k \in \partial \Phi\left(z^k\right) \text { for } k=0,1, \ldots
$$
\item[(iii)] KL property: $\Phi$ is a KL function.
\item[(iv)] A continuity condition: If a subsequence $\left\{z^{k_n}\right\}$ converges to $\bar{z}$ then $\Phi\left(z^{k_n}\right)$ converges to $\Phi(\bar{z})$ as $n$ goes to $\infty$.
\end{itemize}
Then we have $\sum_{k=1}^{\infty} \varphi_k<\infty$, and $\left\{z^k\right\}$ converges to a critical point of $\Phi$.
\end{lemma}

Hence, according to the Assumption \ref{par} and Lemma \ref{lemma}, we prove theorem \ref{conv} as follows. 
\begin{proof}
    Let $x^*$ be a limit point of $\{x^k\}$. From Theorem \ref{sub} we have $x^*$ is a critical point. Define $F^\delta(x, y):=F(x)+\sum_{i=1}^m \frac{\delta_i}{2}\left\|x_i-y_i\right\|^2$. Let $z^k=\left(x^k, x^{k-1}\right)$, $\varphi_k^2=\frac{1}{2}\left\|x^{k+1}-x^k\right\|^2+\frac{1}{2}\left\|x^k-x^{k-1}\right\|^2$. As the generated sequence $\left\{x^k\right\}$ is assumed to be bounded in the following, we verify the conditions of Lemma \ref{lemma} for $F^{\delta}(x^k,x^{k-1})$ with $\delta_i=(\underline{l}+C\bar{l})/2$. 
    \begin{itemize}
        \item[(i)] Sufficient decrease property:

        From \eqref{sumF}, we have
        \begin{equation}
            F(x^{k+1})+
            \underline{l}\Vert x_i^{k+1}-x_i^k\Vert^2
            \leq
            F(x^k)+C\bar{l}\left\|x_i^k-x_i^{k-1}\right\|^2, 
        \end{equation}
        hence, $F^{\delta}(z^k)-F^{\delta}(z^{k+1})\geq(\underline{l}-C\bar{l})\varphi_k^2$.
        \item[(ii)] Boundedness of subgradient:

        Note that
        $$
        \partial_x F^\delta(x, y)=\partial F(x)+\left[\left.\delta_i\left(x_i-y_i\right)\right|_{i=1, \ldots, m}\right],~ {\rm and} ~ \partial_y F^\delta(x, y)=\left[\left.\delta_i\left(y_i-x_i\right)\right|_{i=1, \ldots, m}\right],
        $$
        with the optimality condition we have
        \begin{equation}
            \nabla_ih(x_i^k,\bar{x}_{\neq i}^{k+1})-\nabla_ih(\bar{x}^{k+1,i})+\frac{\alpha_i^k}{\gamma_i^k}(x_i^k-x_i^{k-1})\in\partial_{x_i}(u_i(\bar{x}_i^{k+1},\bar{x}_{\neq i}^{k+1})+\theta_i(\bar{x}_i^{k+1})).
        \end{equation}
        By Assumption \ref{par}, there exist ${s}_i^k \in \partial_{x_i} u_i\left(\bar{x}_i^{k+1},\bar{x}_{\neq i}^{k+1}\right)$ and ${v}_i^k \in \partial \theta_i\left(\bar{x}_i^{k+1}\right)$ such that
        \begin{equation}  
        \nabla_ih(x_i^k,\bar{x}_{\neq i}^{k+1})-\nabla_ih(\bar{x}^{k+1,i})+\frac{\alpha_i^k}{\gamma_i^k}(x_i^k-x_i^{k-1})={s}_i^k+{v}_i^k, 
       \end{equation}  
        and there exists ${t}_i^k \in \partial_{x_i} h\left(x^{k+1}\right)$ such that
        \begin{equation}  
        \left\|{s}_i^k-{t}_i^k\right\| \leq B_i\left\|x^{k+1}-(x_i^k,\bar{x}_{\neq i}^{k+1})\right\| .
        \end{equation}  
        We note that ${t}_i^k+{v}_i^k \in \partial_{x_i} F\left(x^{k+1}\right)$ by Assumption \ref{par}. On the other hand, 
        \begin{equation}   \left\|{t}_i^k+{v}_i^k\right\|=\left\|{t}_i^k-{s}_i^k+{s}_i^k+{v}_i^k\right\| \leq B_i\left\|x^{k+1}-(x_i^k,\bar{x}_{\neq i}^{k+1})\right\|+\frac{2\alpha_i^k}{\gamma_i^k}\Vert x_i^k-x_i^{k-1}\Vert,
        \end{equation}  
        which implies the boundedness of the subgradient.
        \item[(iii)] KL property:

        Since $F$ is a KL function, $F^\delta$ is also a KL function.

        \item[(iv)] A continuity condition:

         Suppose $z^{k_n}\rightarrow z^*$, Lemma \ref{bound} implies that if $x^{k_n} \rightarrow x^*$, then $x^{k_n-1} \rightarrow x^*$. Hence $z^* = (x^*,x^*)$. On the other hand, we know that for $i\in[p]$, $h(x_i^{k_n},\bar{x}_{\neq i}^{k-1})+\theta_i(x_i^{{k_n}})\rightarrow h(x^*)+\theta_i(x_i^*)$. Hence $F(x^{k_n})$ converges to $F(x^*)$, which leads to $F^\delta(z^{{k_n}+1})$ converges to $F^\delta(z^*)$.
    \end{itemize}
    From Lemma \ref{lemma}, the whole generated sequence $\{x^k\}$ of the proposed iBPLM algorithm is convergent. 
\end{proof}

Moreover, if the function $\psi$ appearing in the KL inequality takes the form $\psi(s)=c s^{1-\theta}$ with $\theta \in[0,1)$ and $c>0$, we can derive the convergence rates for both sequences $\{x^k\}$ and $\{F (x^k ) \}$.

\begin{theorem}
Let $\left\{x^k\right\}$ be a sequence generated by iBPLM. Suppose that Assumptions in Theorem \ref{conv} are satisfied. If $\left\{x^k\right\}$ is bounded and the function $\psi$ appearing in the $K L$ inequality takes the form $\psi(s)=c s^{1-\theta}$ with $\theta \in[0,1)$ and $c>0$, then the following statements hold.
(i) If $\theta=0$, the sequences $\left\{x^k\right\}$ and $\left\{F\left(x^k\right)\right\}$ converge in a finite number of steps to $x^*$ and $F^*$, respectively.
(ii) If $\theta \in(0,1 / 2]$, the sequences $\left\{x^k\right\}$ and $\left\{F\left(x^k\right)\right\}$ converge linearly to $x^*$ and $F^*$, respectively.
(iii) If $\theta \in(1 / 2,1)$, there exist positive constants $\delta_1, \delta_2$, and $N$ such that $\left\|x^k-x^*\right\| \leq \delta_1 k^{\frac{1-\theta}{2 \theta-1}}$ and $F\left(x^k\right)-F^* \leq \delta_2 k^{-\frac{1}{2 \theta-1}}$ for all $k \geq N$.
\end{theorem}

Because this theorem can be proven by using the same techniques as those in the proofs of Attouch and Bolte \cite[theorem 2]{attouch2009convergence}, we omit the detail of the proof.
 
\subsection{Plug-and-Play iBPLM Algorithm}\label{pnpalg}
In this section, we propose the following plug-and-play inertial block proximal linearised minimisation (PnP-iBPLM) algorithm~\ref{pnppalm} to solve the non-convex optimisation problem (\ref{objective}). Our approach extends the iBPLM algorithm by integrating deep priors as regularisers, enhancing the robustness of the optimisation process. 

\begin{algorithm}[t!]
\caption{Plug-and-play iBPLM algorithm (PnP-iBPLM)}\label{pnppalm}
\begin{algorithmic}
\STATE{{\bf Initialisation:} Select $w_{ij}\in[0,1]$ for $j<i$ and  $w_{ij}=1$ for $j\geq i$. Given $x_i^0$, set $x_i^{-1}=x_i^{0}$.
}
\FOR{$k=1, 2, 3, \ldots$}
\STATE{Update $x_j^{k,i}= (1-w_{ij})x_j^{k+1}+w_{ij}x_j^k,~~ \forall i,j=1,2,\ldots,p$.}
\FOR{$i=1, 2, 3, \ldots, p$}
\STATE{Choose $\alpha_i^k\in[0,1)$.}
\STATE{Update $\hat{x}_i^k=x_i^k+\alpha_i^k(x_i^k-x_i^{k-1})$;}
\STATE{Update $x_i^{k+1}=\mathcal{D}_{\sigma_i}(\hat{x}_i^{k}-\gamma_i^k \nabla_i h({\hat x}_i^k, x_{\neq i}^{k,i}))$.}
\ENDFOR
\IF{stopping criterion is satisfied,}
\STATE{
Return $(x_1^k, x_2^k, \ldots, x_p^k)$.
}
\ENDIF
\ENDFOR
\end{algorithmic}
\end{algorithm}

We update the $x_i~(i=1,2,\ldots,p)$-subproblems using the gradient step Denoiser \cite{cohen2021has,wu2024extrapolated}, defined by
\begin{equation}\label{denoiser}
\mathcal{D}_{\sigma_i}=I-\nabla g_{\sigma_i}, i=1,2,\ldots,p, 
\end{equation}
which is obtained from a scalar function
\begin{equation}
g_{\sigma_i}=\frac{1}{2}\left\|{x}-N_{\sigma_i}({x})\right\|^2,
\end{equation}
where the mapping $N_{\sigma_i}({x})$ is implemented using a differentiable neural network. This enables the explicit computation of $g_{\sigma_i}$ and ensures that $g_{\sigma_i}$ has a Lipschitz gradient with a constant $L$ (where $L<1$).
Originally, the denoiser $\mathcal{D}_{\sigma_i}$, described in \eqref{denoiser}, is trained to denoise images degraded by Gaussian noise of level $\sigma_i$. Notably, the denoiser $\mathcal{D}_{\sigma_i}$ takes the form of a proximal mapping of a weakly convex function, as detailed in the next proposition.

\begin{proposition}[\cite{wu2024extrapolated}, Proposition 4.1]
$\mathcal{D}_{\sigma_i}({x})=\operatorname{prox}_{\phi_{\sigma_i}}({x})$, where $\phi_{\sigma_i}$ is defined by
\begin{equation}\label{phi_sigma}
\phi_{\sigma_i}({x})=g_{\sigma_i}\left(\mathcal{D}_{\sigma_i}^{-1}({x})\right)-\frac{1}{2}\left\|\mathcal{D}_{\sigma_i}^{-1}({x})-{x}\right\|^2, i=1,2,\ldots, p,
\end{equation}
if ${x} \in \operatorname{Im}\left(\mathcal{D}_{\sigma_i}\right)$, and $\phi_{\sigma_i}({x})=+\infty$ otherwise. Moreover, $\phi_{\sigma_i}$ is $\frac{L}{L+1}$-weakly convex and $\nabla \phi_{\sigma_i}$ is $\frac{L}{1-L}$-Lipschitz on $\operatorname{Im}\left(\mathcal{D}_{\sigma_i}\right)$, and $\phi_{\sigma_i}({x}) \geq g_{\sigma_i}({x}), \forall {x} \in \mathbb{R}^n$.
\end{proposition}

In the following, we present the convergence results of PnP-iBPLM Algorithm.
\begin{theorem}
Let $g_{\sigma_i}: \mathbb{R}^n \rightarrow \mathbb{R} \cup\{+\infty\}$ of class $\mathcal{C}^2$ with L-Lipschitz continuous gradient with $L<1$, and $\mathcal{D}_{\sigma_i}=I-\nabla g_{\sigma_i}$ and $h$ is a lower semi-continuous function. Suppose that $h$ and $g_{\sigma_i}$ are bounded from below, then the sequence $\{x^{k}\}$ generated by the proposed PnP-iBPLM Algorithm, which is assumed to be bounded. Then,
\begin{itemize}
    \item[(i)] for positive parameters $\xi_i^k$ and $\delta_i^k$ satisfy $\xi_i^{k+1}\leq C \delta_i^k$ for some constant $C\in(0,1)$, we know that with $\theta_i:=\phi_{\sigma_i}$, the following holds
    \begin{equation}
F(x^k)+\sum_{i=1}^p\frac{\xi_i^k}{2}\left\|x_i^k-x_i^{k-1}\right\|^2
    \geq
    F({x}^{k+1})+
    \sum_{i=1}^p\frac{\delta_i^k}{2}\Vert x_i^{k+1}-x_i^k\Vert^2, k=0,1, \cdots.
\end{equation}
\item[(ii)] if there exists a positive parameter $\underline{l}$ such that $\min_{i,k}\{\frac{\delta_i^k}{2}\}\geq\underline{l}$, then we have
        \begin{equation}
            \sum_{k=0}^{+\infty} \sum_{i=1}^m\left\|x_i^{k+1}-x_i^k\right\|^2<+\infty.
        \end{equation}
\item[(iii)] every limit point $x^*$ of $\left\{x^k\right\}$ is a critical point. 
\item[(iv)] suppose assumption \ref{par} hold.  Let $F$ be a KL function and together with the existence of $\underline{l}$ in Lemma \ref{bound}, we also assume there exists $\bar{l}>0$, such that $\max_{i,k}\{\frac{\delta_i^k}{2}\}\leq\bar{l}$. For $C$ in Proposition \ref{NSDP} satisfying $C<\underline{l}/\bar{l}$, the whole generated sequence $\{x^k\}$ of the proposed PnP-iBPLM algorithm is convergent. 
\end{itemize}
\end{theorem}
\begin{proof}
    The PnP-iBPLM algorithm is a special case of the problem in~\eqref{objective} with  $\theta_i=\phi_{\sigma_i}$ with given assumptions. Therefore, it follows from Proposition \ref{NSDP}, (i) holds; from Lemma \ref{bound} that (ii) holds;  and from  Theorem \ref{sub} that (iii) is confirmed.
    Conclusion (iv) can be derived from Lemma \ref{lemma} and Theorem \ref{conv}. This completes the proof. 
\end{proof}

\begin{remark}
Note that in the implementation, we can select part of the subproblems solved by the gradient-step denoiser, leading to hybrid methods.
Furthermore, according to \cite[Lemma 1]{hurault2023convergent}, $\phi_{\sigma_i}({x})$ in \eqref{phi_sigma} satisfies the Kurdyka-{\L}ojasiewicz (KL) property if $g_{\sigma_i}$ is real analytic \cite{krantz2002primer} in a neighborhood of $ x\in\mathbb{R}^n$ and its Jacobian matrix $Jg_{\sigma_i} ({ x})$ is nonsingular. The analyic nature of  $g_{\sigma_i}$ can be assured for a wide range of deep neural networks.
Additionally, the nonsingularity of $Jg_{\sigma_i} ({x})$ can be guaranteed by assuming $L<1$ as discussed in \cite{hurault2023convergent}. For more discussions on general conditions under which the KL property holds for deep neural networks, we refer to \cite{castera2021inertial, zeng2019global}. Therefore, selecting a neural network for $g_{\sigma_i}$ that guarantees the KL property of $\phi_{\sigma_i}({x})$ during implementation is attainable.
\end{remark}

\section{Experimental Results}\label{results}
We validate our proposed iBPLM and PnP-iBPLM algorithms on two inverse problem tasks: image denoising and image restoration. We test the algorithms using two blocks in the image denoising task and four blocks in the image restoration task, \textit{both employing dictionary learning-based models}. Following standard protocol, we use two widely recognised metrics: Peak Signal to Noise Ratio (PSNR) and Structural Similarity Index (SSIM). All experiments in this section were run using PyTorch on an NVIDIA RTX A6000 GPU. 

\begin{wraptable}{r}{0.6\textwidth}
\vspace{-0.7cm}
\caption{Denoising results of the proposed framework with $\tau_X=0.5$, and K-SVD for noise levels $\sigma = [25, 50, 75, 100]$.}
\begin{center}
\resizebox{0.6\textwidth}{!}{
\begin{tabular}{ccccccc}
    \hline
    \cellcolor[HTML]{EFEFEF}\textsc{Noise}&\cellcolor[HTML]{EFEFEF}\textsc{Metric}&\cellcolor[HTML]{EFEFEF}\textsc{Degraded} & \cellcolor[HTML]{EFEFEF}K-SVD \cite{rubinstein2012analysis} & \cellcolor[HTML]{EFEFEF}\textcolor{blue}{\FiveStar}iBPLM \\ \hline
    \multirow{2}{*}{$\sigma=25$}&PSNR&20.18&25.58&\cellcolor[HTML]{D7FFD7}25.61\\  
                                &SSIM&82.32&94.08&\cellcolor[HTML]{D7FFD7}94.60\\ \hline
    \multirow{2}{*}{$\sigma=50$}&PSNR&14.16 &21.27&\cellcolor[HTML]{D7FFD7}22.91\\
                                &SSIM&59.75&86.64&\cellcolor[HTML]{D7FFD7}89.73\\ \hline
    \multirow{2}{*}{$\sigma=75$}&PSNR&10.64&18.79&\cellcolor[HTML]{D7FFD7}20.17\\
                                &SSIM&42.64&79.31&\cellcolor[HTML]{D7FFD7}82.20\\ \hline
    \multirow{2}{*}{$\sigma=100$}&PSNR&8.14&17.38&\cellcolor[HTML]{D7FFD7}17.92\\
                                &SSIM&30.84&73.61&\cellcolor[HTML]{D7FFD7}73.78\\                        
    \hline
\end{tabular}
}
\vspace{-0.7cm}

\end{center}
\label{tab:tau}
\end{wraptable}

\subsection{iBPLM for Image Denoising}
For the specific inverse problem \eqref{dic}, following the proposed iBPLM algorithm, we set \( h(D, X) = \frac{1}{2}\|DX - Y\|^2 \), \( \theta_X(X) = \lambda_X \Vert X\Vert_0 \), and \( \theta_D(D) = \lambda_D \Vert D\Vert^2\). This definition is  consistent with the traditional dictionary learning, enabling the dictionary learning denoising model to be solved using our iBPLM algorithm. 
We start by comparing our technique against that of K-SVD \cite{rubinstein2012analysis}.

Table~\ref{tab:tau} provides a comparative analysis of denoising performance between the well-established K-SVD method and our proposed iBPLM framework (with $\tau_X = 0.5$) across a range of noise levels. This analysis underscores the iBPLM framework’s superior performance in enhancing both PSNR and SSIM metrics across all tested scenarios. 
At the lower noise level of $\sigma=25$, the iBPLM not only improves upon the `Degraded' baseline but also edges out K-SVD in SSIM, showcasing its acute capability to preserve image structural integrity even in subtler noise environments. As noise levels escalate, the resilience of iBPLM becomes even more apparent; it consistently records higher SSIM values than K-SVD, which points to its robust ability to maintain visual quality and textural details under more severe noise conditions. This trend holds true even at the very high noise setting of  $\sigma=100$, where iBPLM continues to deliver superior structural preservation as indicated by its higher SSIM scores. The consistent outperformance of iBPLM across various noise intensities highlights its potential as a particularly effective tool for applications demanding high fidelity image restoration. 

\begin{figure}[t!]
\hspace{-0.13in}
 \subfigure[\footnotesize{\small Original}]{
	\zoomincludgraphic{0.224\textwidth}{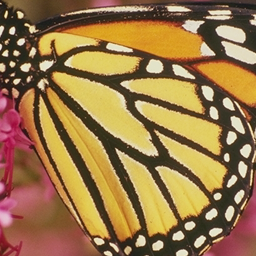}{0.14}{0.42}{0.24}{0.32}{4}{help_grid_off}{up_right}{line_connection_off}{2.5}{green}{1.}{green}
	} \hspace{-0.255in}
    \subfigure[\footnotesize{\small Observed (20.18/82.32)}]{
	\zoomincludgraphic{0.224\textwidth}{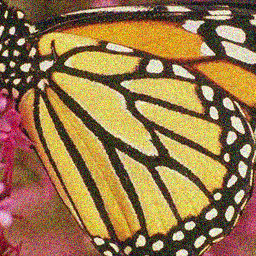}{0.14}{0.42}{0.24}{0.32}{4}{help_grid_off}{up_right}{line_connection_off}{2.5}{green}{1.}{green}
	}\hspace{-0.22in}
    \subfigure[\footnotesize{\small K-SVD (25.58/94.08)}]{
	\zoomincludgraphic{0.224\textwidth}{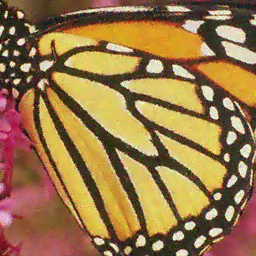}{0.14}{0.42}{0.24}{0.32}{4}{help_grid_off}{up_right}{line_connection_off}{2.5}{green}{1.}{green}
	}\hspace{-0.22in}
    \subfigure[\footnotesize{\small iBPLM (25.63/94.61)}]{
	\zoomincludgraphic{0.224\textwidth}{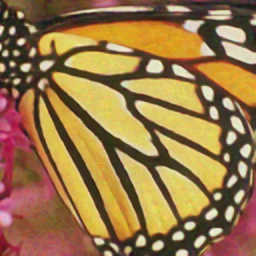}{0.14}{0.42}{0.24}{0.32}{4}{help_grid_off}{up_right}{line_connection_off}{2.5}{green}{1.}{green}
	}\hspace{-0.22in}

\vspace{-0.1in}
\hspace{-0.13in}
 \subfigure[\footnotesize{\small Original}]{
	\zoomincludgraphic{0.224\textwidth}{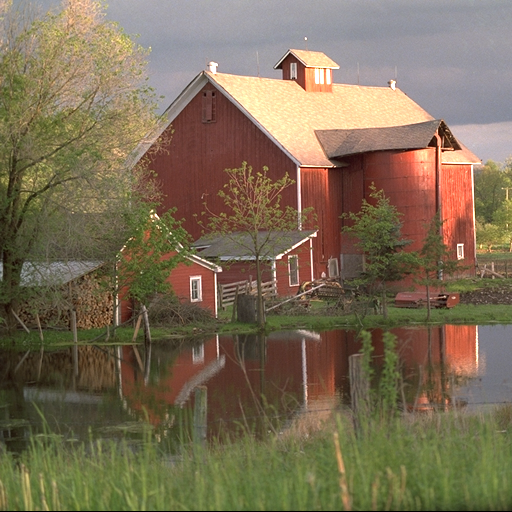}{0.43}{0.385}{0.51}{0.455}{4.5}{help_grid_off}{bottom_right}{line_connection_off}{2.5}{green}{1}{green}	}\hspace{-0.22in}
 \subfigure[\footnotesize{\small Observed (20.17/55.56)}]{
	\zoomincludgraphic{0.224\textwidth}{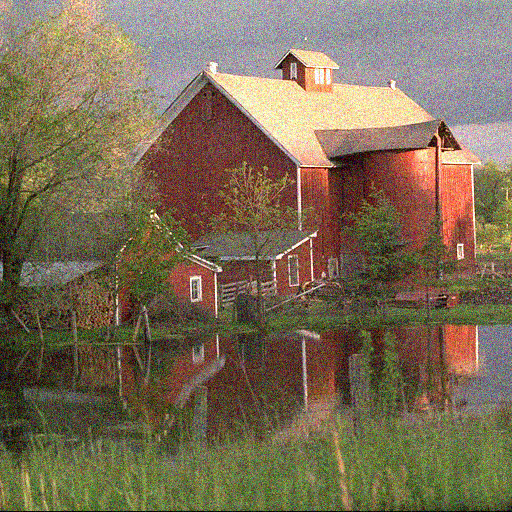}{0.43}{0.385}{0.51}{0.455}{4.5}{help_grid_off}{bottom_right}{line_connection_off}{2.5}{green}{1}{green} 
	}\hspace{-0.22in}
  \subfigure[\footnotesize{\small K-SVD (26.71/84.76)}]{
	\zoomincludgraphic{0.224\textwidth}{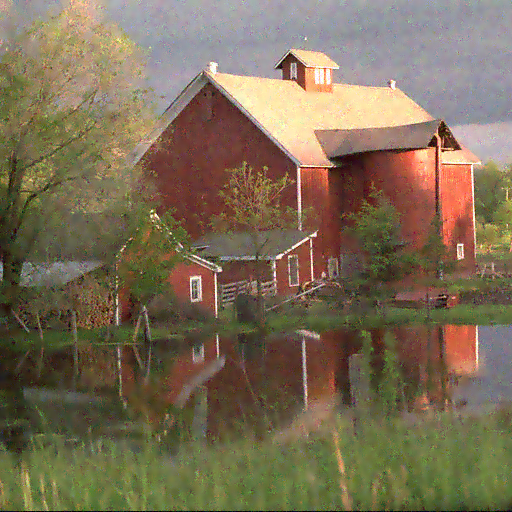}{0.43}{0.385}{0.51}{0.455}{4.5}{help_grid_off}{bottom_right}{line_connection_off}{2.5}{green}{1}{green}  
	}\hspace{-0.22in}
   \subfigure[\footnotesize{\small iBPLM (27.31/85.86)}]{
	\zoomincludgraphic{0.224\textwidth}{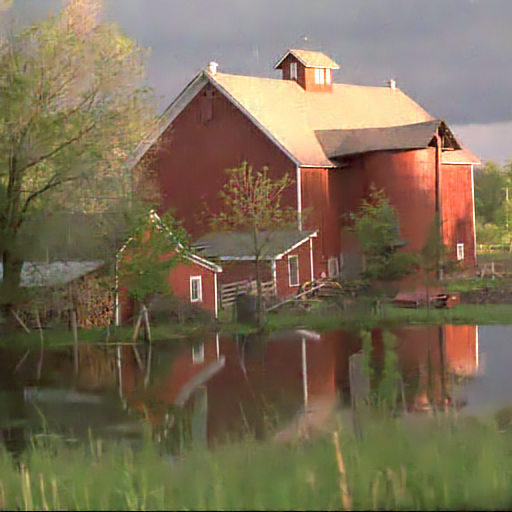}{0.43}{0.385}{0.51}{0.455}{4.5}{help_grid_off}{bottom_right}{line_connection_off}{2.5}{green}{1}{green} 
	}\hspace{-0.22in}
 \caption{Image restoration results (PSNR/SSIM) with Gaussian noise level $25$. 
Visualisation comparing our technique with K-SVD \cite{rubinstein2012analysis}, with two examples presented in (a)-(d) and (e)-(h), respectively.
 }\label{fig: denoise}
 \end{figure}

The visual results in Figure \ref{fig: denoise} further support our iBPLM framework in preserving detail and enhancing image quality compared to K-SVD. For the butterfly image, iBPLM restores intricate wing patterns with greater clarity and less noise, as indicated by the higher PSNR and SSIM scores. In the house image, iBPLM demonstrates its strength in rendering architectural details and textures more sharply, particularly in areas such as the roof and surrounding foliage. These images visually support the numerical findings, showing iBPLM's superior performance in reducing noise while maintaining the integrity of the original images.

\begin{figure}[t]
\centerline{\includegraphics[width=6in]{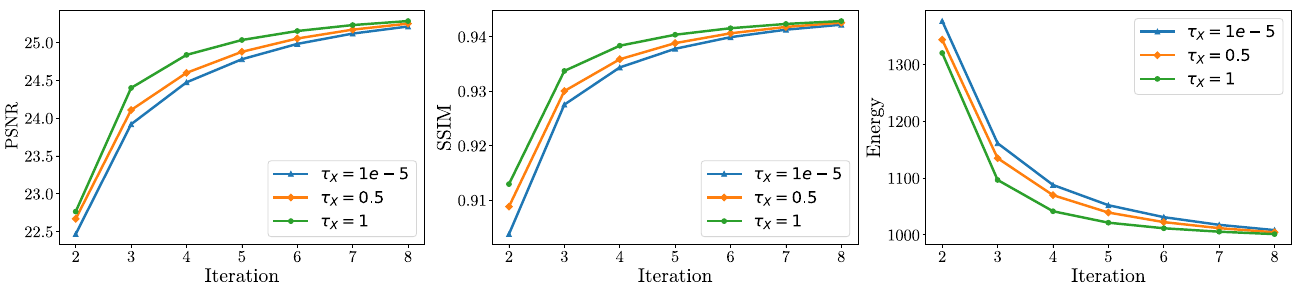}
		}
\caption{PSNR, SSIM, and energy curves of the proposed methods with different $\tau_X$ values from iteration 2 to 8.
 }\label{fig:tau}
 \end{figure}

\begin{figure}[h]
\centerline{\includegraphics[width=6in]{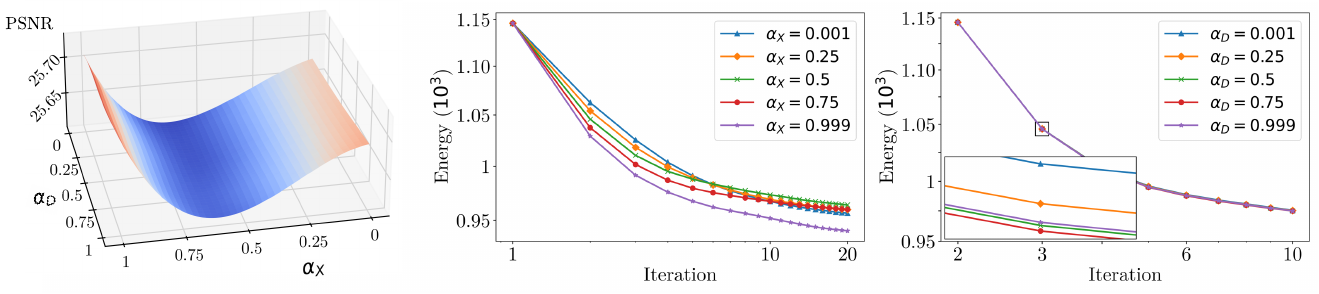}
		}
\caption{Effect of $\alpha_X$ and $\alpha_D$ in iBPLM algorithm on `butterfly' with noise level $25$. 
  The first image is the PSNR surface under different $\alpha_X$ and $\alpha_D$. The second one is the energy curves of different $\alpha_X$ with fixed $\alpha_D=0.5$. The third one is the energy curves of different $\alpha_D$ with fixed $\alpha_X=0.5$. The comparisons of the second and third plots are conducted through log-log scale analysis.
 }\label{fig:alpha}
 \end{figure}
 
Furthermore, some detailed analysis of weight parameter $\tau_X$ and inertial parameters $\alpha_D$ and $\alpha_X$ are provided. 
We plot the PSNR, SSIM, and energy curves of our methods in Figure \ref{fig:tau} to showcase the performance of different weight parameter $\tau_X$ values from iteration 2 to 8. The extrapolated parameters $\alpha_X$ and $\alpha_D$ are analyzed in Figure \ref{fig:alpha}. Specifically, the range of $\alpha_X$ and $\alpha_D$ is set to $[0.001, 0.25, 0.5, 0.75, 0.999]$. The PSNR surface indicates that the final performance is influenced by the extrapolated parameters. Furthermore, we plot the energy curves of inertial parameter $\alpha_X$ (with $\alpha_D = 0.5$ fixed) and $\alpha_D$ (with $\alpha_X = 0.5$ fixed) along the iteration. Specifically, variations in $\alpha_X$ and $\alpha_D$ lead to different energy convergence rates, showcasing the importance of these parameters in refining the model's efficiency and effectiveness. Overall, these results underscore the robustness of our proposed denoising model and the beneficial role of parameter tuning in achieving superior image quality.

\subsection{PnP-iBPLM for Image Restoration}
For this specific inverse problem, given the input image $Y$ with some linear operator $B$, the image restoration model can be formulated as follows:
\begin{equation}\label{model}
    \min_{X,D,I,Z}\frac{\eta}{2}\Vert DX-I\Vert^2+\frac{\beta}{2}\Vert I-Z\Vert^2+\lambda_X\phi_X(X)+\lambda_D\phi_D(D)+\lambda_Z\phi_Z(Z)+\frac{\lambda_I}{2}\Vert BI-Y\Vert^2,
\end{equation}
where $D$ is the dictionary, $X$ is the corresponding sparse coefficient, $I$ is the latent image, and $\eta$, $\beta$, $\lambda_X$, $\lambda_D$, $\lambda_Z$ and $\lambda_I$ are positive parameters. The $\phi_D$, $\phi_X$, and $\phi_Z$ serve as regularisers for $D$, $X$, $Z$, respectively.  
However, we simplify the iteration by applying only one denoiser in $\theta_Z$. For the variables $D$ and $X$ in the dictionary learning method, we use the regularizer according to their definition to better fit the meaning of the model. For $\theta_I$, we use the data-fitting term to constrain the image restoration model. 
Following the proposed PnP-iBPLM algorithm, this model can be solved by setting $h(X,D,I,Z) = \frac{\eta}{2}\Vert DX-I\Vert^2+\frac{\beta}{2}\Vert I-Z\Vert^2$, $\theta_X = \lambda_X\phi_X(X)$, $\theta_D = \lambda_D\phi_D(D)$, $\theta_I = \frac{\lambda_I}{2}\Vert BI-Y\Vert^2$, $\theta_Z = \lambda_Z\phi_Z(Z)$. 
The following are the more specific settings. 

For $X$-subproblem, we have 
\begin{equation}\label{X}
    X^{k+1}\in\underset{X}{\arg\min}~\frac{1}{2\gamma_X}\Big\Vert X-\hat{X}^{k}+ \gamma_X\hat{h}_{X}^{k}\Big\Vert^2+\lambda_X\phi_X(X),
\end{equation}
where $\hat{h}_{X}^k= \eta D^{{k}^{T}}(D^{k}\hat{X}^{k}-I^{k})$, $\hat{X}^{k}=X^{k}+\alpha_X(X^{k}-X^{k-1})$, $D^{{k}^{T}}$ is the conjugate of $D^{k}$, $\alpha_D\in [0,1]$, and $\gamma_X\in(0,1/L_X)$, $L_X$ is the Lipschitz constant. Following the definition of the traditional dictionary learning model, we set $\phi_X(X) = \Vert X\Vert_0$ to describe the sparsity of the coefficients. Hence, the $X$-subproblem can be solved by the hard-shrinkage method. 

For $D$-subproblem, we have 
\begin{equation}\label{D}
    D^{k+1}\in\underset{D}{\arg\min}~\frac{1}{2 \gamma_D}\Big\Vert D-\hat{D}^{k}+\gamma_D\hat{h}_{D}^{k}\Big\Vert^2+\lambda_D\phi_D(D),
\end{equation}
where $\hat{h}_{D}^k = \eta(\hat{D}^{k}\bar{X}^{k}-I^{k})\bar{X}^{k^T}$, $\hat{D}^{k}=D^{k}+\alpha_D(D^{k}-D^{k-1})$, $\bar{X}^k=X^{k}+\tau_X (X^{k+1}- X^{k})$, $\bar{X}^{{k}^{T}}$ is the conjugate of $\bar{X}^k$, $\alpha_D\in [0,1]$, $\tau_X\in [0,1]$, 
and $\gamma_D\in(0,1/L_D)$, $L_D$ is the Lipschitz constant. Following the definition of the traditional dictionary learning model, we set $\phi_D(D) = \Vert D\Vert^2$. Hence, the $D$-subproblem has a close-formed solution. 

For $I$-subproblem, we get
\begin{equation}
    I^{k+1}=\underset{I}{\arg\min}~\frac{1}{2\gamma_I}\Big\Vert I-\hat{I}^{k}+\gamma_I\hat{h}_{I}^{k}\Big\Vert^2+\frac{\lambda_I}{2}\Vert BI-Y\Vert^2,
\end{equation}
with $\hat{h}_{I}^{k}=\eta(\hat{I}^{k}-\bar{D}^{k}\bar{X}^{k})+\gamma(\hat{I}^{k}-Z^{k})$, $\hat{I}^{k}=I^{k}+\alpha_I(I^{k}-I^{k-1})$, $\bar{D}^k=D^{k}+\tau_D (D^{k+1}- D^{k})$, $\bar{X}^k=X^{k}+\tau_X (X^{k+1}- X^{k})$, $\alpha_I\in [0,1]$, $\tau_D\in [0,1]$, $\tau_X\in [0,1]$, and $\gamma_I\in(0,1/L_I)$, $L_I$ is the Lipschitz constant. $I$-subproblem also has a close-formed solution. 

For the $Z$-subproblem 
\begin{equation}
    Z^{k+1}\in\underset{Z}{\arg\min}~\frac{1}{2\gamma_Z}\Big\Vert Z-\hat{Z}^{k}+\gamma_Z\hat{h}_{Z}^{k}\Big\Vert^2+\lambda_Z\phi_Z(Z),
\end{equation}
where $\hat{h}_Z^k = \beta(\hat{Z}^{k}-\bar{I}^k)$, $\hat{Z}^{k} = Z^{k}+\alpha_Z(Z^{k}-Z^{k-1})$, $\bar{I}^k=I^{k}+\tau_I (I^{k+1}- I^{k})$, $\tau_I\in [0,1]$, and $\gamma_Z\in(0,1/L_Z)$, $L_Z$ is the Lipschitz constant. Hence it can be solved by a deep denoiser 
\begin{equation}\label{Z}
    Z^{k+1}=\mathcal{D}_\sigma\Big(\hat{Z}^{k}-\gamma_Z\hat{h}_{Z}^{k}, \sqrt{\gamma_Z\lambda_Z}\Big).
\end{equation}
More specifically, we use the classical DRUNet \cite{zhang2021plug} as our deep gradient step denoiser. DRUNet incorporates both U-Net and ResNet architectures and takes an additional noise level map as input, achieving state-of-the-art performance in Gaussian noise removal. Similar to the setting in \cite{wu2024extrapolated,hurault2022gradient}, we regularize the training loss of $\mathcal{D}_\sigma$ using the spectral norm $\Vert\cdot\Vert_S$ of the Hessian of $g_\sigma$ as follows
\begin{equation}
    \mathcal{L}_S(\sigma)=\mathbb{E}_{\mathbf{x} \sim p, \xi_\sigma \sim \mathcal{N}\left(0, \sigma^2\right)}\left[\left\|\mathcal{D}_\sigma\left(\mathbf{x}+\xi_\sigma\right)-\mathbf{x}\right\|^2+0.01 \max \left(\left\|\nabla^2 g_\sigma\left(\mathbf{x}+\xi_\sigma\right)\right\|_S, 0.9\right)\right]
\end{equation}
to ensure the Lipschitz constant of $\nabla g_\sigma$ is less than $1$, which is consistent with Section \ref{pnpalg}. Berkeley segmentation dataset, Waterloo Exploration Database, DIV2K dataset, and Flick2K dataset are applied as the training sets. 

Note that we follow the setting in the \cite{wu2024extrapolated,hurault2022gradient}, only noise level $\{2.55, 7.65, 12.75\}$ are considered in training. After training, we apply the pre-trained deep gradient step neural network as the denoiser to handle the image restoration problem of heavy Gaussian noise. Experiments show that our method can even handle the image corrupted with heavy motion blur and Gaussian noise.

\begin{table}[t!]
\caption{Average PSNR (dB) and SSIM ($\%$) results of different restoration models for MB$(20, 60)/\sigma=25$. We refer to 'Equivariant' as 'Equi.', and denote our approach with a \textcolor{blue}{\FiveStar}.}
\begin{center}
\resizebox{1\textwidth}{!}
{\begin{tabular}{ccccccccc}
    \hline
     \cellcolor[HTML]{EFEFEF}Datasets&\cellcolor[HTML]{EFEFEF}Metric&\cellcolor[HTML]{EFEFEF}Degraded & \cellcolor[HTML]{EFEFEF}DPIR~\cite{zhang2021plug} & \cellcolor[HTML]{EFEFEF}DiffPIR~\cite{zhu2023denoising} & \cellcolor[HTML]{EFEFEF}Equi.~\cite{terris2023equivariant} &  \cellcolor[HTML]{EFEFEF}SNORE \cite{renaud2024plug} &\cellcolor[HTML]{EFEFEF}DYSdiff \cite{wu2024extrapolated}&\cellcolor[HTML]{EFEFEF}\textcolor{blue}{\FiveStar}PnP-iBPLM\\
    \hline
    \multirow{2}{*}{Set3C}&PSNR&14.75&20.78 & 20.64 & 22.82 &22.14&22.67&\cellcolor[HTML]{D7FFD7}23.05\\
                          &SSIM&50.85& 84.67& 85.41 & 90.68&86.95&89.46&\cellcolor[HTML]{D7FFD7}90.71\\
    \hline
    \multirow{2}{*}{CBSD10}&PSNR&17.58& 24.14& 23.24 & 24.74 &24.01&24.51&\cellcolor[HTML]{D7FFD7}24.77 \\
                          &SSIM&31.08& 74.27& 70.97 &75.70&70.91&74.94&\cellcolor[HTML]{D7FFD7}76.17\\ 
    \hline
    \multirow{2}{*}{Set17}&PSNR&17.79& 24.16& 22.76 & {24.65} & 24.17&24.46&\cellcolor[HTML]{D7FFD7}24.74 \\
                          &SSIM&42.20& 77.75& 73.86& 80.42& 77.08&79.88&\cellcolor[HTML]{D7FFD7}80.85\\  
    \hline
    \multirow{2}{*}{Set18}&PSNR&18.21& 25.56& 22.86 & \cellcolor[HTML]{D7FFD7}26.25 &25.51&26.03&{26.22}\\
                          &SSIM&51.49& 83.76& 79.47& \cellcolor[HTML]{D7FFD7}87.94 &84.15&86.81&{87.50}\\ 
    \hline
    \multirow{2}{*}{Kodak24}&PSNR&18.24& 24.99& 24.23& 	25.25 &24.70&25.16&\cellcolor[HTML]{D7FFD7}25.50 \\
                            &SSIM&36.30& 79.97& 77.51& \cellcolor[HTML]{D7FFD7}81.19 &75.73&79.79&{81.03} \\ 
    \hline
\end{tabular}}
\end{center}
\label{tab:restoration}
\end{table}

\begin{figure}[!b]
\hspace{-0.18in}
 \subfigure[\footnotesize{\small Original}]{
	\zoomincludgraphic{0.223\textwidth}{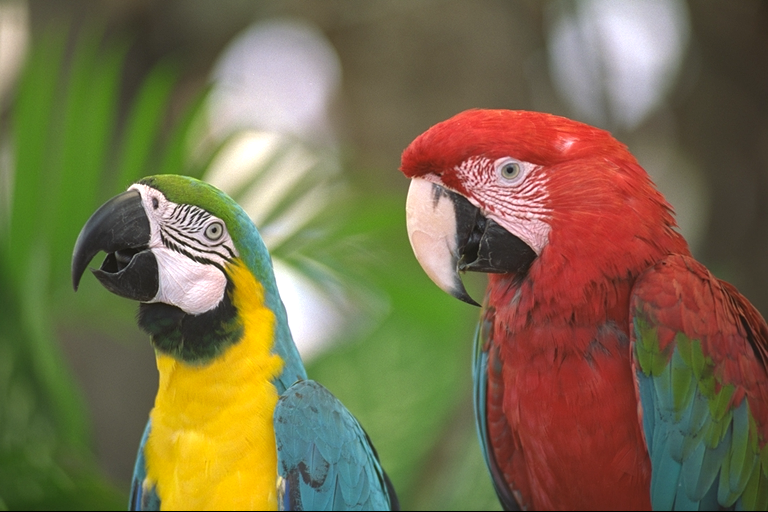}{0.205}{0.615}{0.31}{0.45}{2.7}{help_grid_off}{up_right}{line_connection_off}{2.5}{green}{1.5}{green} 
	} \hspace{-0.255in}
    \subfigure[\footnotesize{\small Observed (19.18/52.78)}]{
	\zoomincludgraphic{0.223\textwidth}{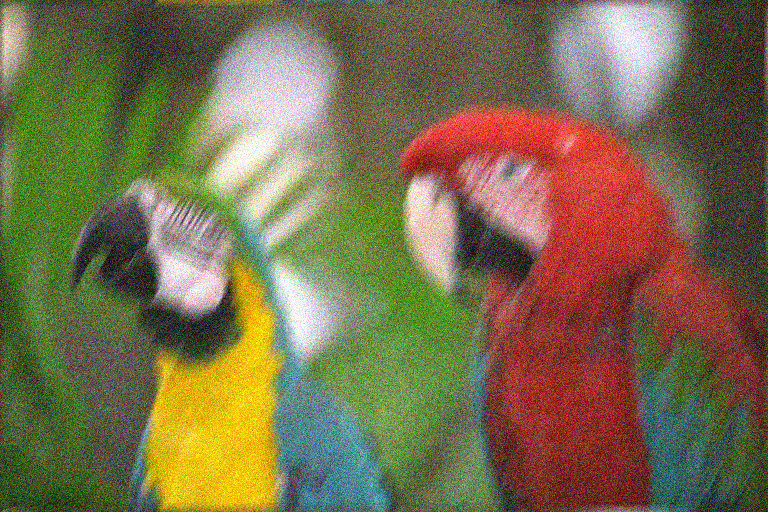}{0.205}{0.615}{0.31}{0.45}{2.7}{help_grid_off}{up_right}{line_connection_off}{2.5}{green}{1.5}{green} 
	}\hspace{-0.22in}
    \subfigure[\footnotesize{\small DPIR (28.13/94.60)}]{
	\zoomincludgraphic{0.223\textwidth}{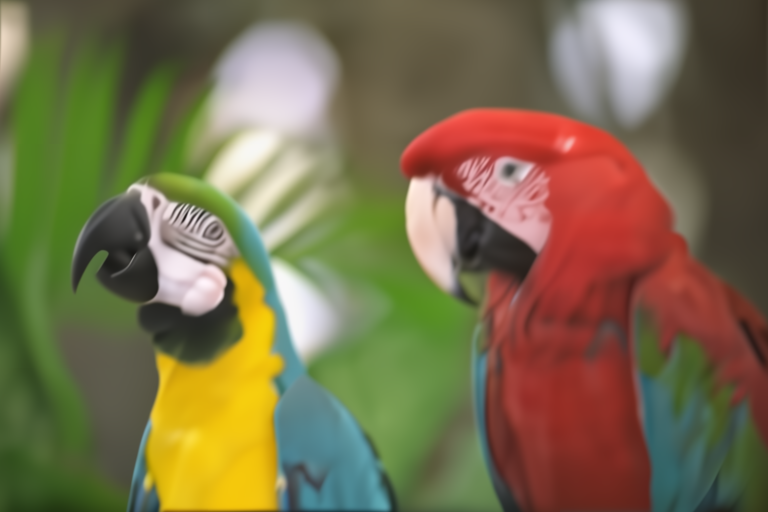}{0.205}{0.615}{0.31}{0.45}{2.7}{help_grid_off}{up_right}{line_connection_off}{2.5}{green}{1.5}{green}
	}\hspace{-0.22in}
    \subfigure[\footnotesize{\small DiffPIR (26.95/93.12)}]{
	\zoomincludgraphic{0.223\textwidth}{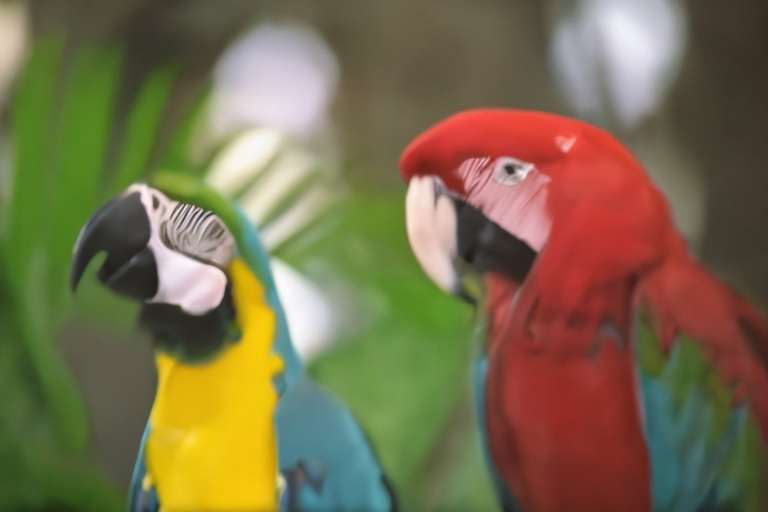}{0.205}{0.615}{0.31}{0.45}{2.7}{help_grid_off}{up_right}{line_connection_off}{2.5}{green}{1.5}{green} 
	}\hspace{-0.22in}

\vspace{-0.1in}
\hspace{-0.18in}
 \subfigure[\footnotesize{\small Equi.  (28.66/95.17)}]{
	\zoomincludgraphic{0.223\textwidth}{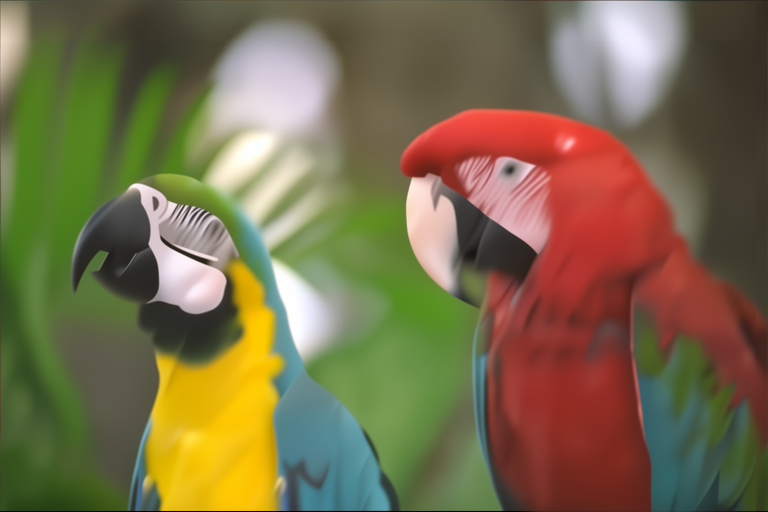}{0.205}{0.615}{0.31}{0.45}{2.7}{help_grid_off}{up_right}{line_connection_off}{2.5}{green}{1.5}{green} 
}\hspace{-0.22in}
 \subfigure[\footnotesize{\small SNORE (27.26/89.88)}]{
	\zoomincludgraphic{0.223\textwidth}{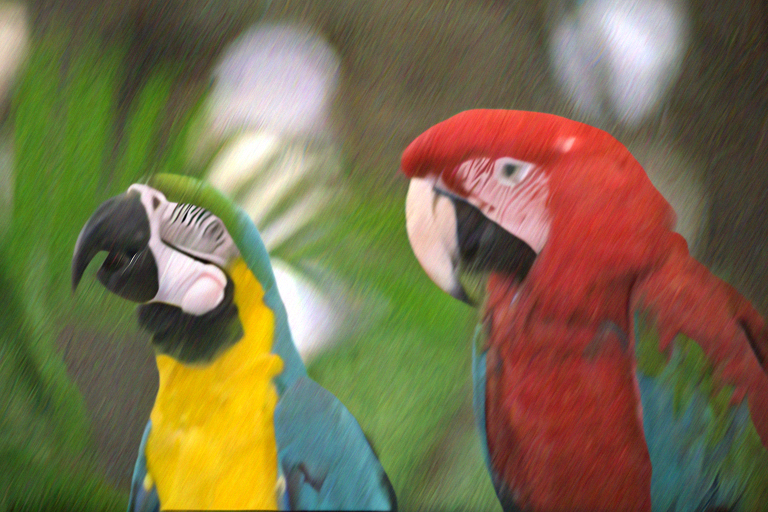}{0.205}{0.615}{0.31}{0.45}{2.7}{help_grid_off}{up_right}{line_connection_off}{2.5}{green}{1.5}{green} 
	}\hspace{-0.22in}
  \subfigure[\footnotesize{\small DYSdiff (28.18/93.78)}]{
	\zoomincludgraphic{0.223\textwidth}{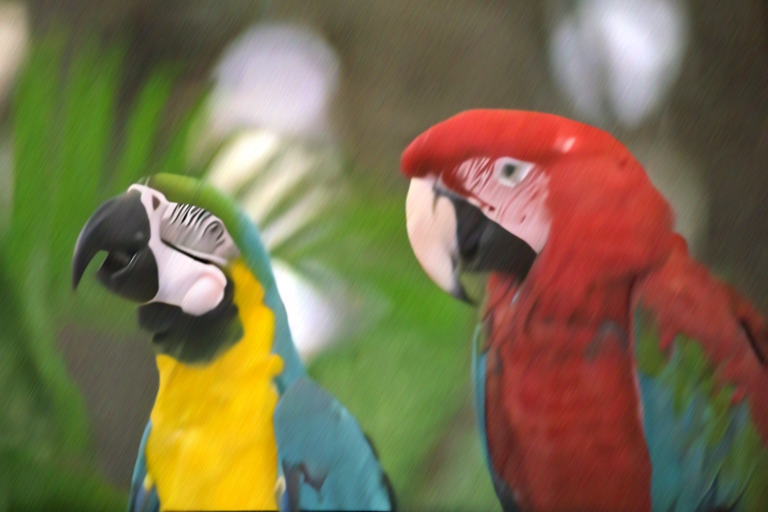}{0.205}{0.615}{0.31}{0.45}{2.7}{help_grid_off}{up_right}{line_connection_off}{2.5}{green}{1.5}{green}  
	}\hspace{-0.22in}
   \subfigure[\footnotesize{\small Ours (29.06/94.75)}]{
	\zoomincludgraphic{0.223\textwidth}{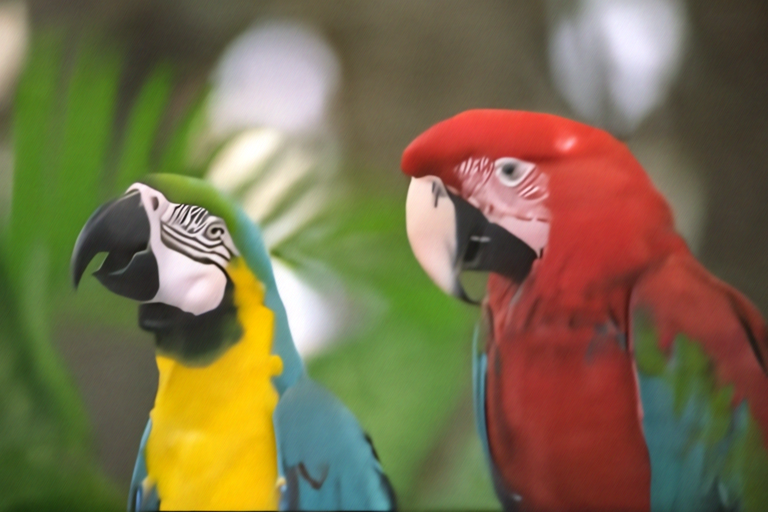}{0.205}{0.615}{0.31}{0.45}{2.7}{help_grid_off}{up_right}{line_connection_off}{2.5}{green}{1.5}{green} 
	}\hspace{-0.22in}
 \caption{Image restoration results (PSNR/SSIM) with motion blur kernel MB$(20, 60)$ and Gaussian noise level $25$. We refer to 'Equivariant' as 'Equi.'. Visualisation comparison of our scheme and some state-of-the-art PnP-based methods: (c) DPIR~\cite{zhang2021plug}, (d) DiffPIR~\cite{zhu2023denoising}, (e) Equivariant~\cite{terris2023equivariant}, (f) SNORE \cite{renaud2024plug}, (g) DYSdiff \cite{wu2024extrapolated}, and (h) Our PnP-iBPLM.
 }\label{fig: parrot}
 \end{figure}

\begin{figure}[h]
	\begin{minipage}{0.245\linewidth}
		\centering	\centerline{\includegraphics[width=1.35in]{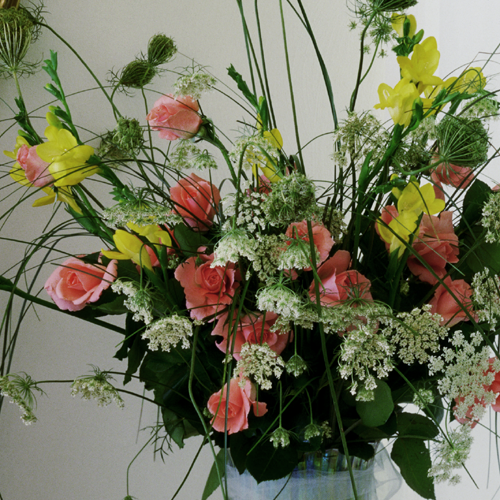}
		}\vspace{-0.03in}
		\centerline{\small (a) Original}
	\end{minipage}  
        \begin{minipage}{0.245\linewidth}
		\centering
	\centerline{\includegraphics[width=1.35in]{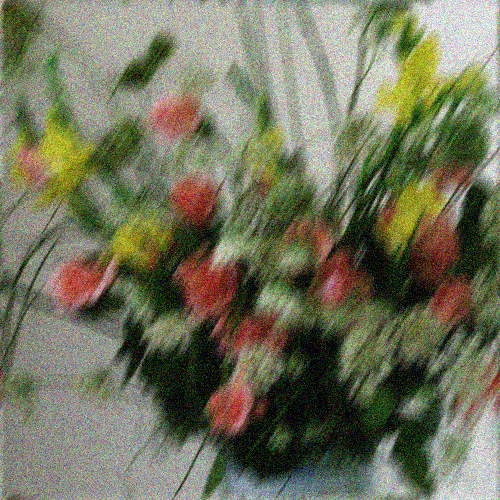}
		}\vspace{-0.03in}
		\centerline{\small (b) Degraded (16.42/30.82)}
	\end{minipage}  
        \begin{minipage}{0.245\linewidth}
		\centering
	\centerline{\includegraphics[width=1.35in]{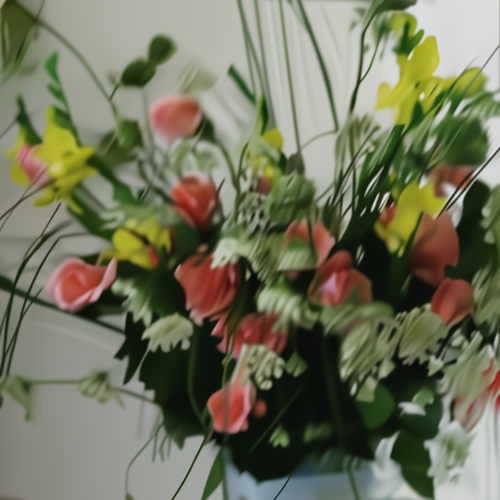}
		}\vspace{-0.03in}
		\centerline{\small (c) DPIR (21.22/71.09)}
	\end{minipage}  
 \begin{minipage}{0.245\linewidth}
		\centering
	\centerline{\includegraphics[width=1.35in]{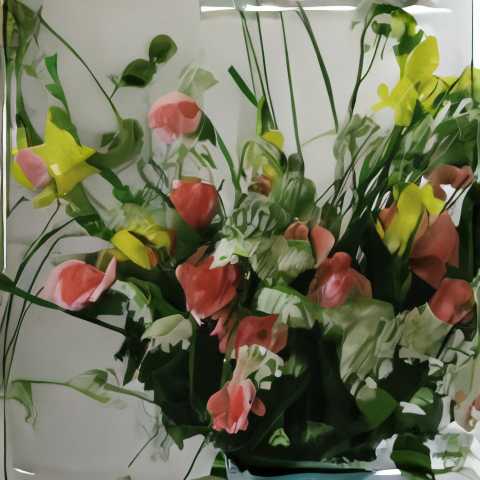}
		}\vspace{-0.03in}
		\centerline{\small (d) DiffPIR (19.31/64.25)}
	\end{minipage}  
  \begin{minipage}{0.245\linewidth}
		\centering
	\centerline{\includegraphics[width=1.35in]{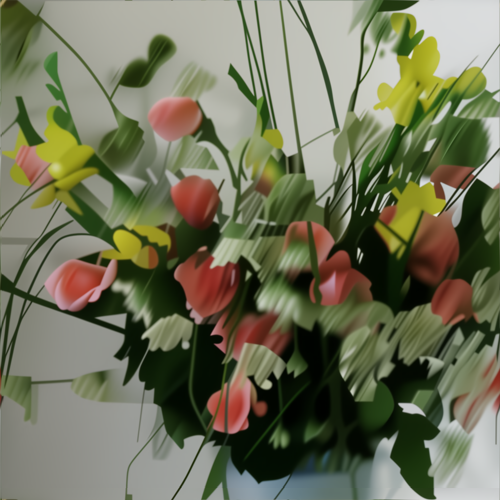}
		}\vspace{-0.03in}
		\centerline{\small (e) Equi. (21.62/75.47)  }
	\end{minipage}
  \begin{minipage}{0.245\linewidth}
		\centering
	\centerline{\includegraphics[width=1.35in]{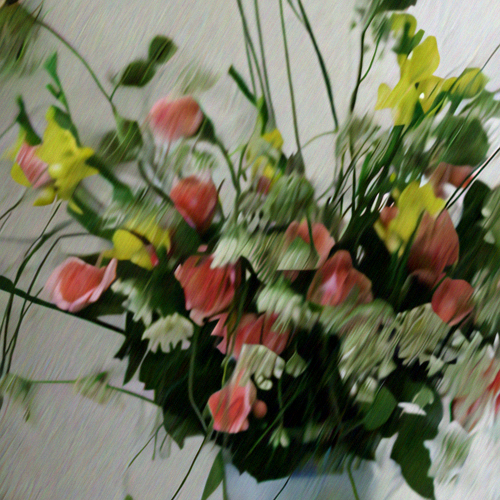}
		}\vspace{-0.03in}
		\centerline{\small (f) SNORE (21.47/70.52) }
	\end{minipage} 
  \begin{minipage}{0.245\linewidth}
		\centering
	\centerline{\includegraphics[width=1.35in]{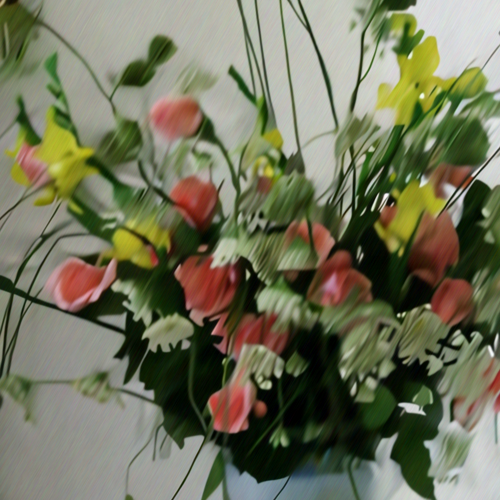}
		}\vspace{-0.03in}
		\centerline{\small (g) DYSdiff (21.68/74.35)}
	\end{minipage} 
 \begin{minipage}{0.245\linewidth}
		\centering
	\centerline{\includegraphics[width=1.35in]{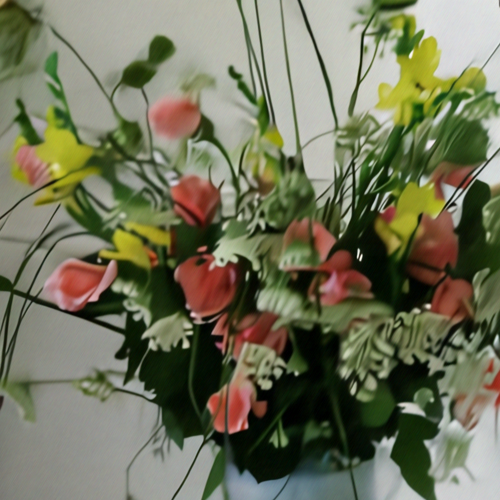}
		}\vspace{-0.03in}
		\centerline{\small (h) Ours (21.92/76.11)}
	\end{minipage}  
\vspace{-0.2cm} \caption{Image restoration results (PSNR/SSIM) with motion blur kernel MB$(20, 60)$ and Gaussian noise level $25$. We refer to 'Equivariant' as 'Equi.'. Visualisation comparison of our scheme and some state-of-the-art PnP-based methods: (c) DPIR~\cite{zhang2021plug}, (d) DiffPIR~\cite{zhu2023denoising}, (e) Equivariant~\cite{terris2023equivariant}, (f) SNORE \cite{renaud2024plug}, (g) DYSdiff \cite{wu2024extrapolated}, and (h) Our PnP-iBPLM.
 }\label{fig:3}
 \end{figure}

\begin{figure}[!b]
	\begin{minipage}{0.245\linewidth}
		\centering	\centerline{\includegraphics[width=1.35in]{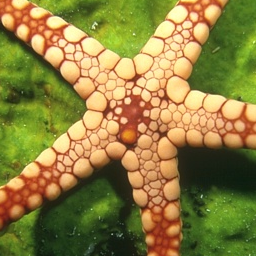}
		}\vspace{-0.03in}
		\centerline{\small (a) Original}
	\end{minipage}  
        \begin{minipage}{0.245\linewidth}
		\centering
	\centerline{\includegraphics[width=1.35in]{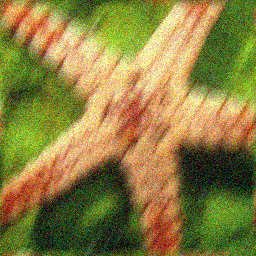}
		}\vspace{-0.03in}
		\centerline{\small (b) Degraded (17.12/68.42)}
	\end{minipage}  
        \begin{minipage}{0.245\linewidth}
		\centering
	\centerline{\includegraphics[width=1.35in]{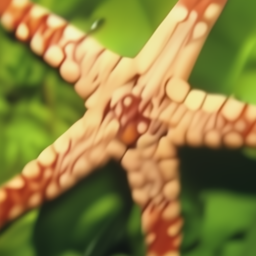}
		}\vspace{-0.03in}
		\centerline{\small (c) DPIR (22.51/92.11) }
	\end{minipage}  
 \begin{minipage}{0.245\linewidth}
		\centering
	\centerline{\includegraphics[width=1.35in]{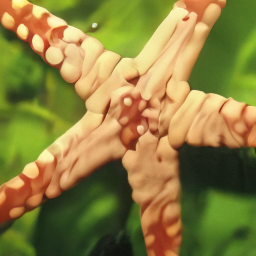}
		}\vspace{-0.03in}
		\centerline{\small (d) DiffPIR (21.80/90.73) }
	\end{minipage}  
  \begin{minipage}{0.245\linewidth}
		\centering
	\centerline{\includegraphics[width=1.35in]{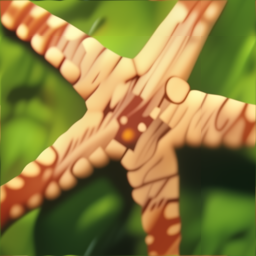}
		}\vspace{-0.03in}
		\centerline{\small (e) Equi. (22.82/92/78)}
	\end{minipage}
  \begin{minipage}{0.245\linewidth}
		\centering
	\centerline{\includegraphics[width=1.35in]{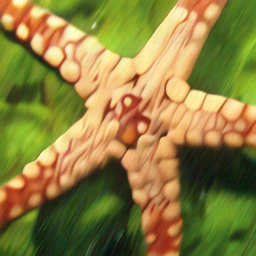}
		}\vspace{-0.03in}
		\centerline{\small (f) SNORE (22.78/91.28)}
	\end{minipage} 
  \begin{minipage}{0.245\linewidth}
		\centering
	\centerline{\includegraphics[width=1.35in]{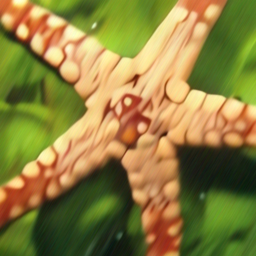}
		}\vspace{-0.03in}
		\centerline{\small (g) DYSdiff (22.90/92.60)}
	\end{minipage} 
 \begin{minipage}{0.245\linewidth}
		\centering
	\centerline{\includegraphics[width=1.35in]{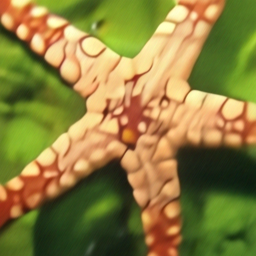}
		}\vspace{-0.03in}
		\centerline{\small (h) Ours (23.12/93.01)}
	\end{minipage}  
\vspace{-0.2cm} \caption{Image restoration results (PSNR/SSIM) with motion blur kernel MB$(20, 60)$ and Gaussian noise level $25$. We refer to 'Equivariant' as 'Equi.'. Visualisation comparison of our scheme and some state-of-the-art PnP-based methods: (c) DPIR~\cite{zhang2021plug}, (d) DiffPIR~\cite{zhu2023denoising}, (e) Equivariant~\cite{terris2023equivariant}, (f) SNORE \cite{renaud2024plug}, (g) DYSdiff \cite{wu2024extrapolated}, and (h) Our PnP-iBPLM.
 }\label{fig:6}
 \end{figure}

We begin by comparing our proposed PnP-iBPLM against the existing techniques of that DPIR~\cite{zhang2021plug}, DiffPIR~\cite{zhu2023denoising}, Equivariant~\cite{terris2023equivariant}, SNORE \cite{renaud2024plug}, DYSdiff \cite{wu2024extrapolated}.
Table \ref{tab:restoration} illustrates the performance of our proposed PnP-iBPLM technique across various datasets at a noise level of $\sigma$, comparing favorably with established restoration models like DPIR, DiffPIR, Equivariant, and SNORE. Notably, PnP-iBPLM consistently achieves top-tier SSIM scores, supporting its exceptional capability in preserving image structure and texture, particularly evident in the CBSD10, Set17, and Kodak24 datasets where it leads with the highest SSIM values. While its PSNR scores are occasionally outperformed by the Equivariant model, PnP-iBPLM demonstrates robust overall effectiveness in PSNR as well, especially highlighted in Set18 and Kodak24. This performance underscores the algorithm’s utility in producing high-quality restorations across different types of images and conditions, confirming its adaptability and strength in handling complex noise levels and various degradation types. 

The visual results depicted in Figures \ref{fig: parrot}, \ref{fig:3}, and \ref{fig:6} further support the quantitative findings discussed earlier, affirming the superior performance of PnP-iBPLM in image restoration tasks. Clear and crisp details are preserved in the denoised images produced by PnP-iBPLM, showcasing its ability to effectively remove noise while retaining essential image features. Comparative analysis against existing techniques such as DPIR, DiffPIR, Equivariant, SNORE, and DYSdiff further solidifies PnP-iBPLM's position as a state-of-the-art solution for high-fidelity image restoration. 

\section{Conclusion}\label{conclusion}
In this paper, we proposed the inertial block proximal linearised minimisation (iBPLM) algorithm, a novel framework for effectively solving non-convex inverse problems in image processing. Addressing the limitations of existing optimisation algorithms, our approach bridges the gap between block-based and alternating-based methods, combining their strengths while mitigating their weaknesses. By integrating extrapolation techniques and deep denoisers within the framework, our PnP-iBPLM algorithm demonstrates robustness and effectiveness in handling non-convex and non-smooth problems, supported by theoretical analysis. Our work contributes significantly to the fields of optimisation and image processing, providing a novel solution for non-convex inverse problems. The superior performance of the iBPLM and PnP-iBPLM algorithms in image denoising and restoration tasks underscores their potential for a wide range of applications.

\bibliographystyle{siamplain}
\bibliography{ref}

\end{document}